\documentclass[12pt,reqno]{amsproc}%
\usepackage{amsfonts}
\usepackage{amsmath}
\usepackage{amssymb}
\usepackage{graphicx}%
 \theoremstyle{plain}

\newtheorem{corollary}{Corollary}[section]

\newtheorem{definition}{Definition}[section]
\newtheorem{example}{Example}[section]

\newtheorem{lemma}{Lemma}[section]
\newtheorem{notation}{Notation}[section]

\newtheorem{remark}{Remark}[section]

\newtheorem{theorem}{Theorem}[section]
 \numberwithin{equation}{section}
\begin{document}

\centerline{\Large\bf{Topological Free  Entropy Dimensions  in
Nuclear C$^*$-algebras }}

\vspace{1cm}

\centerline{\Large\bf{  and in Full Free Products of Unital
C$^*$-algebras}}

\vspace{1cm}

\centerline{Don Hadwin \qquad \qquad  Qihui Li\qquad \qquad  and
\qquad \qquad  Junhao Shen\footnote{The author is partially
supported by an NSF grant.}}

\bigskip

\noindent\textbf{Abstract: } In the paper, we introduce a new
concept, topological orbit dimension, of $n$-tuple of elements  in a
unital C$^*$-algebra. Using this concept, we conclude that
Voiculescu's topological free entropy dimension of every finite
family of self-adjoint generators of a nuclear C$^*$-algebra is less
than or equal to $1$. We also show that the Voiculescu's topological
free entropy dimension is additive in the full free product of some
unital C$^*$-algebras. In the appendix, we show that unital  full
free product of Blackadar and
Kirchberg's unital  MF 
  algebras is also   MF algebra. As an application, we
obtain  that $Ext(C_r^*(F_2)*_{\Bbb C}C_r^*(F_2)) $ is not a
group.

\vspace{0.2cm} \noindent{\bf Keywords:} Topological free entropy
dimension, unital C$^*$-algebra

\vspace{0.2cm} \noindent{\bf 2000 Mathematics Subject
Classification:} Primary 46L10, Secondary 46L54

\section{Introduction}

The theory of free probability and free entropy was developed by
Voiculescu from 1980s. His theory plays a crucial role in the recent
study of finite von Neumann algebras (see \cite{BDK}, \cite{Dyk},
\cite{Ge1}, \cite{Ge2}, \cite{GePo}, \cite{GS1}, \cite{GS2},
\cite{HaSh}, \cite{Jung2}, \cite{JuSh}, \cite{V2}, \cite{V3},
\cite{V4}).  The notion of topological free entropy dimension of
$n$-tuple of elements in a unital C$^*$-algebra, as an  analogue of
free entropy dimension for finite von Neumann algebras in
C$^*$-algebra context, was also introduced by Voiculescu in
\cite{Voi}, where basic properties of free entropy dimension are
discussed.

We started our investigation on   properties of topological free
entropy dimension in \cite {HaSh2}, where we computed the
topological free entropy dimension of a self-adjoint element in a
unital C$^*$-algebra. Some estimations of topological free entropy
dimensions in   infinite dimensional, unital, simple C$^*$-algebras
with a unique trace, which include irrational rotation
C$^*$-algebra, UHF algebra and $C^*_{red}(F_2)\otimes_{min}
C^*_{red}(F_2)$, were also obtained in the same paper. In
\cite{HaSh3}, we proved a formula of topological free entropy
dimension in the orthogonal sum (or direct sum) of unital
C$^*$-algebras. As a corollary, we computed the
  topological free entropy dimension of every finite family of self-adjoint generators of a finite dimensional
C$^*$-algebra.
 In this
article, we will continue our investigation on the concept of
Voiculescu's  topological free entropy dimension.

To study Voiculecu's topological free entropy dimension, firstly
we introduce a notion of topological orbit dimension $\frak
K_{top}^{(2)}$, a modification of ``topological free orbit
dimension" in \cite{HaSh2} which is   inspired by the paper
\cite{HaSh}, of $n$-tuple of self-adjoint elements in a unital
C$^*$-algebra. We prove that $\frak K_{top}^{(2)}$ is a
C$^*$-algebra invariant. In fact we have the following result.

\vspace{0.2cm}
 \noindent{\bf Theorem 3.2: } Suppose that $\mathcal A$ is a unital C$^*$-algebra and
$\{x_1,\ldots, x_n\}$, $\{y_1,\ldots, y_p\}$ are two families of
self-adjoint generators of $\mathcal A$. Then
$$
\frak K_{top}^{(2)}(x_1,\ldots, x_n) = \frak
K_{top}^{(2)}(y_1,\ldots, y_p).
$$

\vspace{0.2cm}
 Moreover, we show in {\bf Theorem 3.1} that if $z_1,\ldots,z_m$ is
 a family of self-adjoint elements in a unital C$^*$-algebra
 $\mathcal A$, then
$$
\delta_{top}(z_1,\ldots,z_m)\le \max\{ \frak
K_{top}^{(2)}(z_1,\ldots,z_m),1\},
$$ where $
\delta_{top}(z_1,\ldots,z_m)$ is the Voiculescu's topological free
entropy dimension of $z_1,\ldots,z_m$ in $\mathcal A$. This
result, together with Theorem 3.2, provides us a possible way to
compute the Voiculescu's topological free entropy dimension of an
arbitrary family of self-adjoint generators of a unital
C$^*$-algebra $\mathcal A$  by studying the topological orbit
dimension of a specific family of self-adjoint generators in
$\mathcal A$.

We then study the topological orbit dimension in  tensor products
and orthogonal sums of unital C$^*$-algebras and obtain the
following results.

\vspace{0.2cm}


\noindent {\bf Theorem 3.4: }  Suppose that $\mathcal A$ is a unital
C$^*$-algebra and $n$ is a positive integer. Suppose that $\mathcal
B= \mathcal A\otimes \mathcal M_n(\Bbb C)$. If $x_1,\ldots, x_m$ is
a family of self-adjoint generators of $\mathcal A$ and $y_1,\ldots,
y_p$ is a family of self-adjoint generators of $\mathcal B$, then
$$
  \frak K_{top}^{(2)}(y_1,\ldots, y_p) \le \frak
  K_{top}^{(2)}(x_1,\ldots, x_m).
$$

\noindent {\bf Theorem 3.5: } Suppose that $\mathcal A$ and
$\mathcal B$ are  unital C$^*$-algebras with a family of
self-adjoint generators $x_1,\ldots, x_n$, and $y_1,\ldots, y_m$
respectively. Suppose $\mathcal D= \mathcal A\bigoplus \mathcal B$
is the orthogonal sum of $\mathcal A$ and $\mathcal B$ with a family
of self-adjoint generators $z_1,\ldots, z_p$. Then,
$$
\frak K_{top}^{(2)}(z_1,\ldots, z_p) \le \frak
K_{top}^{(2)}(x_1,\ldots, x_n)+ \frak K_{top}^{(2)}(y_1,\ldots,
y_m).
$$

\vspace{0.2cm}

Next we show that   topological orbit dimension is always majorized
by   orbit dimension capacity:

\vspace{0.2cm} \noindent{\bf Theorem 4.1: }  Suppose that $\mathcal
A$ is a unital C$^*$-algebra and $x_1,\ldots, x_n$ is a family of
self-adjoint elements in $\mathcal A$. Then
$$
\frak K_{top}^{(2)}(x_1,\ldots, x_n) \le \frak K\frak
K_{2}^{(2)}(x_1,\ldots, x_n),
$$ where $\frak K\frak
K_{2}^{(2)}(x_1,\ldots, x_n)$ is the orbit dimension capacity
defined in Definition 4.1.

 \vspace{0.2cm}

A direct consequence  of Theorem 4.1 is the computation of
topological free entropy dimension of every finite  family of
self-adjoint generators in a nuclear C$^*$-algebra.

 \vspace{0.2cm}
{\noindent \bf Corollary 4.1: } Suppose $\mathcal A$ is a   unital
nuclear C$^*$-algebra with a family of self-adjoint generators
$x_1,\ldots, x_n$. If $\mathcal A$ has the approximation property
in the sense of Definition 4.2, then
$$
\frak K_{top}^{(2)}(x_1,\ldots, x_n)=0 \qquad \text { and } \qquad
\delta_{top}(x_1,\ldots, x_n) \le 1.
$$

 \vspace{0.2cm}

The lower bound of   topological free entropy dimension of a family
of self-adjoint generators of a nuclear C$^*$-algebra depends on the
choice of nuclear algebras. For example,
$\delta_{top}(x_1,\ldots,x_n)$ $=0$ if $x_1,\ldots, x_n$ is a family
of self-adjoint generators of the unitization of C$^*$-algebra of
compact operators (see Theorem 5.6 in \cite {HaSh2}). On the other
hand, $\delta_{top}(x_1,\ldots,x_n)=1$ if $x_1,\ldots, x_n$ is a
family of self-adjoint generators of a UHF algebra (see Theorem 5.4
in \cite{HaSh2}).

More applications of Theorem 4.1 can be found in Corollary 4.2,
Corollary 4.3 and Corollary 4.4.

The last part of the paper is devoted to prove that the topological
free entropy dimension is additive in unital full free products of
unital C$^*$-algebras with the approximation property in the sense
of Definition 4.2, or equivalently in unital full free products of
unital Blackadar and Kirchberg's MF algebras  (See Theorem 5.1). As
a corollary of Theorem 5.1, we obtain the following result, which is
a generalization of an earlier result proved by Voiculescu in
\cite{Voi}.

\vspace{0.2cm} {\noindent {\bf Corollary 5.1: } Suppose that
$\mathcal A_i$ ($i=1,2,\ldots, m$) is a unital C$^*$-algebra
generated by a self-adjoint element $x_i$ in $\mathcal A_i$. Let
$\mathcal D$ be the unital full free product of $\mathcal A_1,
\ldots \mathcal A_n$ equipped with    unital embedding from each
$\mathcal A_i$ into $\mathcal D$. If we identify the element $x_i$
in $\mathcal A_i$ with its image in $\mathcal D$, then
$$
\delta_{top}(x_1,\ldots, x_n) = \sum_{i=1}^n
\delta_{top}(x_i)=n-\sum_{i=1}^n \frac 1 {n_i},
$$ where $n_i$ is the number of elements in the spectrum of
$x_i$ in $\mathcal A_i$ (We use the notation $1/\infty=0$).

The concept of MF algebras was introduced by Blackadar and Kirchberg
in \cite{BK}. This     class of C$^*$-algebras   plays an important
role in the classification of C$^*$-algebras and it is connected to
 Brown, Douglas and Fillmore's extension theory (see the
striking result of Haagerup and Thorbjensen on $Ext(C_r^*(F_2)$). In
the appendix, we show that the unital full free product of two
Blackadar and Kirchberg's separable unital MF  C$^*$-algebras is
again an MF algebra (See Theorem 5.4). Based on Haagerup and
Thorbj{\o}nsen's work on $Ext(C_r^*(F_2) )$, we are able to conclude
that $Ext(C_r^*(F_2)*_{\Bbb C}C_r^*(F_2)) $ is not a group. This
result provides us new example  of C$^*$-algebra whose extension
semigroup is not a group.

\vspace{0.2cm}

 The organization of the paper is as follows. In
section 2, we give the definitions of topological free entropy
dimension and topological orbit dimension of $n$-tuple of elements
in a unital C$^*$-algebra. Some properties of topological orbit
dimension are discussed in section 3. In section 4, we introduce the
concept of orbit dimensional capacity and discuss its application in
the computations of   topological orbit dimension in finitely
generated nuclear C$^*$ algebras and several other classes of unital
C$^*$-algebras. In section 5, we prove that  topological free
entropy dimension is additive in  unital full free products of some
unital C$^*$ algebras. In the appendix, we show that the unital full
free product of two   MF algebras is again an MF algebra and
$Ext(C_r^*(F_2)*_{\Bbb C}C_r^*(F_2)) $ is not a group.

\section{Definitions  and preliminary}

In this section, we are going to recall Voiculescu's definition of
  topological free entropy dimension of $n$-tuple of elements in
a unital C$^*$-algebra and give the definition of topological
orbit dimension of $n$-tuple of elements in a unital
C$^*$-algebra.
\subsection{A Covering    of  a set  in a metric space}

Suppose $(X,d)$ is a metric space and $K$ is a subset of $X$. A
family of balls in $X$  is called a covering of $K$ if the union of
these balls covers $K$ and the centers of these balls lie in $K$.

\subsection{Covering numbers in  complex matrix algebra
$(\mathcal{M}_{k}(\mathbb{C}))^n$}
 Let $\mathcal{M}_{k}(\mathbb{C})$ be the $k\times k$ full matrix
algebra with entries in $\mathbb{C}$,  and $\tau_{k}$ be the
  normalized trace on $\mathcal{M}_{k}(\mathbb{C})$, i.e.,
  $\tau_{k}=\frac{1}{k}Tr$, where $Tr$ is the usual trace on
  $\mathcal{M}_{k}(\mathbb{C})$.
 Let $\mathcal{U}(k)$ denote the group
of all unitary matrices in $\mathcal{M}_{k}(\mathbb{C})$. Let
$\mathcal{M}_{k}(\mathbb{C})^{n}$ denote the direct sum of $n$
copies of $\mathcal{M}_{k}(\mathbb{C})$.  Let $\mathcal
M_k^{s.a}(\Bbb C)$ be the subset of $\mathcal M_k(\Bbb C)$
consisting of all self-adjoint matrices of $\mathcal M_k(\Bbb C)$.
Let $(\mathcal M_k^{s.a}(\Bbb C))^n$ be the direct sum (or
orthogonal sum) of $n$ copies of $\mathcal M_k^{s.a}(\Bbb C)$. Let
$\|\cdot \|$ be the operator norm  on
  $\mathcal{M}_{k}(\mathbb{C})^{n}$ defined by
  \[
  \Vert(A_{1},\ldots,A_{n})\Vert =  \max\{\|A_1\|,\ldots, \|A_n\| \}    \]
  for all $(A_{1},\ldots,A_{n})$ in $\mathcal{M}_{k}(\mathbb{C})^{n}$.
  Let $\Vert\cdot\Vert_{2}$
  denote the trace norm induced by $\tau_{k}$ on
  $\mathcal{M}_{k}(\mathbb{C})^{n}$, i.e.,
  \[
  \Vert(A_{1},\ldots,A_{n})\Vert_{2} =\sqrt{\tau_{k}(A_{1}^{\ast}A_{1})+\ldots
  +\tau_{k}(A_{n}^{\ast}A_{n})}
  \]
  for all $(A_{1},\ldots,A_{n})$ in $\mathcal{M}_{k}(\mathbb{C})^{n}$.

For every $\omega>0$, we define the $\omega$-$\|\cdot\|$-ball
$Ball_{}(B_{1},\ldots ,B_{n};\omega,  \|\cdot\|)$ centered at
$(B_{1},\ldots,B_{n})$ in $\mathcal{M}_{k}(\mathbb{C})^{n}$ to be
the subset of $\mathcal{M}_{k}(\mathbb{C})^{n}$   consisting of all
$(A_{1},\ldots,A_{n})$ in $\mathcal{M}_{k}(\mathbb{C})^{n}$ such
that $$\Vert(A_{1},\ldots,A_{n})-(B_{1},\ldots,B_{n})\Vert
<\omega.$$

\begin{definition}
Suppose that $\Sigma$ is a subset of
$\mathcal{M}_{k}(\mathbb{C})^{n}$. We define the covering number
$\nu_\infty(\Sigma, \omega)$ to be the minimal number of
$\omega$-$\|\cdot\|$-balls that constitute a covering of $\Sigma$ in
$\mathcal{M}_{k}(\mathbb{C})^{n}$.
\end{definition}

For every $\omega>0$, we define the $\omega$-$\|\cdot\|_2$-ball
$Ball_{}(B_{1},\ldots ,B_{n};\omega,  \|\cdot\|_2)$ centered at
$(B_{1},\ldots,B_{n})$ in $\mathcal{M}_{k}(\mathbb{C})^{n}$ to be
the subset of $\mathcal{M}_{k}(\mathbb{C})^{n}$   consisting of all
$(A_{1},\ldots,A_{n})$ in $\mathcal{M}_{k}(\mathbb{C})^{n}$ such
that $$\Vert(A_{1},\ldots,A_{n})-(B_{1},\ldots,B_{n})\Vert_2
<\omega.$$

\begin{definition}
Suppose that $\Sigma$ is a subset of
$\mathcal{M}_{k}(\mathbb{C})^{n}$. We define the covering number
$\nu_2(\Sigma, \omega)$ to be the minimal number of
$\omega$-$\|\cdot\|_2$-balls that constitute a covering of $\Sigma$
in $\mathcal{M}_{k}(\mathbb{C})^{n}$.
\end{definition}

\subsection{Unitary orbits of balls in $\mathcal M_k(\Bbb C)^n$}



\quad For every $\omega>0$, we define the
$\omega$-orbit-$\|\cdot\|_2$-ball $\mathcal U(B_{1},\ldots
,B_{n};\omega, \|\cdot\|_2)$ centered at $(B_{1},\ldots,B_{n})$ in
$\mathcal{M}_{k}(\mathbb{C})^{n}$ to be the subset of
$\mathcal{M}_{k}(\mathbb{C})^{n}$   consisting of all
$(A_{1},\ldots,A_{n})$ in $\mathcal{M}_{k}(\mathbb{C})^{n}$ such
that there exists a unitary matrix $W$ in $\mathcal{U}(k)$
satisfying
$$\Vert(A_{1},\ldots,A_{n})-(WB_{1}W^*,\ldots,WB_{n}W^*)\Vert_2
<\omega.$$

\begin{definition}
Suppose that $\Sigma$ is a subset of
$\mathcal{M}_{k}(\mathbb{C})^{n}$. We define the covering number
$o_2(\Sigma, \omega)$ to be the minimal number of
$\omega$-orbit-$\|\cdot\|_2$-balls that constitute a covering of
$\Sigma$ in $\mathcal{M}_{k}(\mathbb{C})^{n}$.
\end{definition}

\subsection{Noncommutative polynomials}
In this article, we always assume that $\mathcal A$ is a unital
C$^*$-algebra. Let $x_1,\ldots, x_n, y_1,\ldots, y_m$ be
self-adjoint elements in $\mathcal A$. Let $\Bbb C\langle
X_1,\ldots, X_n, Y_1,\ldots,Y_m\rangle $ be the set of all
noncommutative polynomials in the indeterminates $X_1,\ldots, X_n,
Y_1,\ldots,Y_m$. Let $\{P_r\}_{r=1}^\infty$ be the collection of all
noncommutative polynomials in $\Bbb C\langle X_1,\ldots, X_n,
Y_1,\ldots,Y_m\rangle $ with rational complex coefficients. (Here
``rational complex coefficients" means that the real   and imaginary
parts of all coefficients of $P_r$ are rational numbers).

\begin{remark}
   We alsways assume that $1\in  \Bbb C\langle
X_1,\ldots, X_n, Y_1,\ldots,Y_m\rangle $.
\end{remark}

\subsection{Voiculescu's norm-microstates Space}
For all integers $r, k\ge 1$, real numbers $R, \epsilon>0$ and
noncommutative polynomials $P_1,\ldots, P_r$, we define
$$
\Gamma^{(top)}_R(x_1,\ldots, x_n, y_1,\ldots, y_m; k,\epsilon,
P_1,\ldots, P_r)
$$ to be the subset of $(\mathcal M_k^{s.a}(\Bbb C))^{n+m}$
consisting of all these $$ (A_1,\ldots, A_n, B_1,\ldots, B_m)\in
(\mathcal M_k^{s.a}(\Bbb C))^{n+m}
$$ satisfying
 $$  \max \{\|A_1\|, \ldots, \|A_n\|, \|B_1\|,\ldots, \|B_m\|\}\le R
 $$ and
$$
|\|P_j(A_1,\ldots, A_n,B_1,\ldots,B_m)\|-\| P_j(x_1,\ldots,
x_n,y_1,\ldots,y_m)\||\le \epsilon, \qquad \forall \ 1\le j\le r.
$$



Define the norm-microstates space of $x_1,\ldots,x_n$ in the
presence of $y_1,\ldots, y_m$, denoted by
$$
\Gamma^{(top)}_R(x_1,\ldots, x_n: y_1,\ldots, y_m; k,\epsilon,
P_1,\ldots, P_r),
$$ to be the projection of $
\Gamma^{(top)}_R(x_1,\ldots, x_n, y_1,\ldots, y_m; k,\epsilon,
P_1,\ldots, P_r) $ onto the space $(\mathcal M_k^{s.a}(\Bbb C))^n$
via the mapping
$$
(A_1,\ldots, A_n, B_1,\ldots, B_m) \rightarrow (A_1,\ldots, A_n).
$$

\subsection{Voiculescu's topological free entropy dimension (see \cite{Voi})}

Define
$$
\nu_\infty(\Gamma^{(top)}_R(x_1,\ldots, x_n: y_1,\ldots, y_m;
k,\epsilon, P_1,\ldots, P_r),\omega)
$$ to be the covering number of the set $
\Gamma^{(top)}_R(x_1,\ldots, x_n: y_1,\ldots, y_m; k,\epsilon,
P_1,\ldots, P_r) $ by $\omega$-$\| \cdot\|$-balls in the metric
space $(\mathcal M_k^{s.a}(\Bbb C))^n$ equipped with operator norm.

\begin{definition}Define
$$
  \begin{aligned}
     \delta_{top}(x_1,\ldots, & \ x_n: y_1,\ldots, y_m;  \omega)\\
     & =
     \sup_{R>0} \ \inf_{\epsilon>0, r\in \Bbb N}  \ \limsup_{k\rightarrow\infty} \frac {\log(\nu_\infty(\Gamma^{(top)}_R(x_1,
     \ldots, x_n: y_1,\ldots, y_m; k,\epsilon, P_1,\ldots,
P_r),\omega))}{-k^2\log\omega}\\ \end{aligned}
$$

\noindent {\bf The topological free entropy dimension } of
$x_1,\ldots, x_n$ in the presence of $y_1,\ldots, y_m$ is defined by
$$ \delta_{top}(x_1,\ldots, \ x_n: y_1,\ldots, y_m)=
 \limsup_{\omega\rightarrow 0^+}  \delta_{top}(x_1,\ldots,   x_n:y_1,\ldots, y_m;  \omega)
$$
\end{definition}
\begin{remark} Let $R>
\max\{\|x_1\|,\ldots,\|x_n\|,\|y_1\|,\ldots,\|y_m\|\}$ be a positive
number. By definition, we know
$$
  \begin{aligned}
     \delta_{top}(x_1,\ldots, & \ x_n: y_1,\ldots, y_m )\\
     & =
       \limsup_{\omega\rightarrow 0^+}  \inf_{\epsilon>0, r\in \Bbb N}  \ \limsup_{k\rightarrow\infty} \frac
       {\log(\nu_\infty(\Gamma^{(top)}_R(x_1,
     \ldots, x_n: y_1,\ldots, y_m; k,\epsilon, P_1,\ldots,
P_r),\omega))}{-k^2\log\omega}\\ \end{aligned}
$$
\end{remark}

\subsection{Topological orbit  dimension $\frak K_{top}^{(2)}$}
Define
$$
o_2(\Gamma^{(top)}_R(x_1,\ldots, x_n: y_1,\ldots, y_m; k,\epsilon,
P_1,\ldots, P_r),\omega)
$$ to be the covering number of the set $
\Gamma^{(top)}_R(x_1,\ldots, x_n: y_1,\ldots, y_m; k,\epsilon,
P_1,\ldots, P_r) $ by $\omega$-orbit-$\| \cdot\|_2$-balls in the
metric space $(\mathcal M_k^{s.a}(\Bbb C))^n$ equipped with the
  trace norm.

\begin{definition}Define
$$
  \begin{aligned}
     \frak K_{top}^{(2)}(x_1,\ldots, & \ x_n: y_1,\ldots, y_m;  \omega)\\
     & =
     \sup_{R>0} \ \inf_{\epsilon>0, r\in \Bbb N}  \ \limsup_{k\rightarrow\infty} \frac {\log(o_2(\Gamma^{(top)}_R(x_1,
     \ldots, x_n: y_1,\ldots, y_m; k,\epsilon, P_1,\ldots,
P_r),\omega))}{k^2}\\ \end{aligned}
$$
\begin{remark}
The value of $ \frak K_{top}^{(2)}(x_1,\ldots,  \ x_n: y_1,\ldots,
y_m; \omega)$ increases as $\omega$ decreases.
\end{remark}

\noindent {\bf The topological orbit dimension } of $x_1,\ldots,
x_n$ in the presence of $y_1,\ldots, y_m$ is defined by
$$ \frak K_{top}^{(2)}(x_1,\ldots, \ x_n: y_1,\ldots, y_m)=
 \limsup_{\omega\rightarrow 0^+}  \frak K_{top}^{(2)}(x_1,\ldots,   x_n:y_1,\ldots, y_m;  \omega)
$$
\end{definition}
\begin{remark} In the notation $ \frak K_{top}^{(2)}$, the
subscript ``$top$"  stands for  the norm-microstates space and
superscript ``$(2)$" stands for the using of
unitary-orbit-$\|\cdot\|_2$-balls when counting the covering numbers
of the norm-microstates spaces.
\end{remark}
\begin{remark} Let $R>
\max\{\|x_1\|,\ldots,\|x_n\|,\|y_1\|,\ldots,\|y_m\|\}$ be a positive
number. By definition, we know
$$
  \begin{aligned}
     \frak K_{top}^{(2)}(x_1,\ldots, & \ x_n: y_1,\ldots, y_m )\\
     & =
       \limsup_{\omega\rightarrow 0^+}  \inf_{\epsilon>0, r\in \Bbb N}  \ \limsup_{k\rightarrow\infty} \frac
       {\log(o_2(\Gamma^{(top)}_R(x_1,
     \ldots, x_n: y_1,\ldots, y_m; k,\epsilon, P_1,\ldots,
P_r),\omega))}{k^2}\\ \end{aligned}
$$
\end{remark}

\subsection{C$^*$-algebra ultraproducts}

Suppose $\{\mathcal M_{k_m}(\Bbb C)\}_{m=1}^\infty$ is a sequence of
complex matrix algebras where $k_m$ goes to infinity as $m$ goes to
infinity. Let $\gamma$ be a free ultrafilter in $\beta(\Bbb
N)\setminus \Bbb N$. We can introduce a unital C$^*$-algebra
$\prod_{m=1}^\infty \mathcal M_{k_m}(\Bbb C)$ as follows:
$$
\prod_{m=1}^\infty \mathcal M_{k_m}(\Bbb C) = \{(Y_m)_{m=1}^\infty \
| \ \forall \ m\ge 1, \ Y_m \in \mathcal M_{k_m}(\Bbb C) \ \text {
and } \ \sup_{m\ge 1} \| Y_m\|<\infty\}.
$$
We can also introduce a  norm   closed two sided ideal  $\mathcal
I_\infty$  as follows.
$$
  \begin{aligned}
    \mathcal I_\infty & = \{(Y_m)_{m=1}^\infty\in  \prod_{m=1}^\infty \mathcal M_{k_m}(\Bbb C) \
| \ \lim_{m\rightarrow \gamma } \|Y_m\| =0\}
  \end{aligned}
$$

\begin{definition}
The C$^*$-algebra ultraproduct of $\{\mathcal M_{k_m}(\Bbb
C)\}_{m=1}^\infty$ along the ultrfilter $\gamma$, denoted by
$\prod_{m=1}^\gamma \mathcal M_{k_m}(\Bbb C)$, is defined to be the
quotient
  algebra,
$\prod_{m=1}^\infty \mathcal M_{k_m}(\Bbb C)/\mathcal I_\infty$,
of $\prod_{m=1}^\infty \mathcal M_{k_m}(\Bbb C)$ by the ideal
$\mathcal I_\infty$. The image of $(Y_m)_{m=1}^\infty\in
\prod_{m=1}^\infty \mathcal M_{k_m}(\Bbb C)$ in
$\prod_{m=1}^\gamma \mathcal M_{k_m}(\Bbb C)$ is denoted by
$[(Y_m)_{m}]$.
\end{definition}

\section{Properties of  topological orbit dimension $\frak K_{top}^{(2)}$}
In this section, we are going to discuss  properties of the
topological orbit dimension  $\frak K_{top}^{(2)}$.

The following result explains the relationship between Voiculescu's
topological free entropy dimension and topological orbit dimension
of $n$-tuple of elements in a unital C$^*$-algebra.

\begin{lemma}
Suppose that $\mathcal A$ is a unital C$^*$-algebra and
$x_1,\ldots,x_n$ is a family of self-adjoint elements in $\mathcal
A$. If
$$
 \frak K_{top}^{(2)}(x_1,\ldots, x_n)<\infty,
$$ then
$$
\delta_{top}(x_1,\ldots,x_n)\le 1.
$$

\end{lemma}

\begin{proof}  Let $\{P_r\}_{r=1}^\infty$ be the collection of all
noncommutative polynomials in $\Bbb C\langle X_1,\ldots, X_n\rangle
$ with rational complex coefficients.
 For any
$0<\omega<1/10$, $R>\max\{\|x_1\|,\ldots, \|x_n\|\}$, by Remark 2.3
we know that
$$
\inf_{r\in \Bbb N, \ \epsilon>0} \ \limsup_{k\rightarrow\infty}
\frac {\log(o_2(\Gamma_R^{(top)}(x_1,\ldots, x_n,k,\epsilon,
P_1,\ldots,P_r),\omega))}{-k^2\log\omega}
 \le \frak K_{top}^{(2)}(x_1,\ldots,
x_n) \cdot \frac 1 {-\log\omega}.
$$
By Szarek's results in \cite{Sza}, there is a family of unitary
matrices $\{U_\lambda\}_{\lambda\in\Lambda}$ in $\mathcal U(k)$ such
that (i) $\{Ball(U_\lambda; \frac{\omega}R,\|\cdot
\|)\}_{\lambda\in\Lambda}$ is a covering of $\mathcal U(k)$; and
(ii) the cardinality of $\Lambda$, $|\Lambda|\le \left
(\frac{CR}{\omega}\right)^{k^2}$ where $C$ is a constant independent
of $k, \omega$. Thus from relationship between covering number (see
Definition 2.2) and unitary orbit covering number (see Definition
2.3), we have
$$\begin{aligned}
&\inf_{r\in \Bbb N, \ \epsilon>0} \  \limsup_{k\rightarrow\infty}
\frac {\log(\nu_2(\Gamma_R^{(top)}(x_1,\ldots, x_n,k,\epsilon,
P_1,\ldots,P_r),3\omega))}{-k^2\log\omega}\\
 & \qquad \qquad \quad\le \inf_{r\in \Bbb N, \ \epsilon>0} \ \limsup_{k\rightarrow\infty}
 \frac {\log(o_2(\Gamma_R^{(top)}(x_1,\ldots, x_n,k,\epsilon, P_1,\ldots,P_r),\omega)\cdot \left(\frac {CR}{  \ \omega}  \right)^{k^2})}{-
 k^2\log\omega
 }\\
  & \qquad \qquad \quad\le 1+ \frac {\log C  +\log R}{-\log\omega}+\frak K_{top}^{(2)}(x_1,\ldots,
x_n) \cdot \frac 1 {-\log\omega}.
   \end{aligned}
$$
Now the result of the theorem follows directly from the definitions
of the topological free entropy dimension and the topological orbit
dimension, together with Remark 2.2, Remark 2.5 and the remark in
Section 6 of \cite{Voi} (or Proposition 5.1 in \cite{HaSh2}).

\end{proof}

A direct consequence of the preceding lemma is the following
theorem.
\begin{theorem}
Suppose that $\mathcal A$ is a unital C$^*$-algebra and
$x_1,\ldots,x_n$ is a family of self-adjoint elements in $\mathcal
A$. Then
$$
\delta_{top}(x_1,\ldots,x_n)\le \max\{\frak
K_{top}^{(2)}(x_1,\ldots, x_n),1\}.
$$
In particular, if
$$
\frak K_{top}^{(2)}(x_1,\ldots, x_n)=0,
$$ then
$$
\delta_{top}(x_1,\ldots,x_n)\le 1.
$$
\end{theorem}

The following lemma will be needed in the proof of   Theorem 3.2.
\begin{lemma}
Let $x_{1},\ldots,x_{n},y_{1},\ldots,y_{p}$ be self-adjoint elements
in a unital C$^*$-algebra $\mathcal A$. If $y_{1},\ldots,y_{p}$ are
in the C$^*$ subalgebra generated by $x_{1},\ldots,x_{n}$ in
$\mathcal{A}$, then, for every $\omega>0 ,$
\[
\mathfrak{K}_{top}^{(2)}(x_{1},\ldots,x_{n};4\omega)\le\mathfrak{K}_{top}^{(2)}(x_{1},\ldots,x_{n}%
:y_{1},\ldots,y_{p};2\omega)\le \mathfrak{K}_{top}^{(2)}(x_{1},\ldots,x_{n};\omega)\mathfrak{.}%
\]

\end{lemma}
\begin{proof}
\textbf{\ } It is a straightforward adaptation of the proof of Prop.
1.6 in \cite{V3}. Suppose that $\{P_r\}_{r=1}^\infty$, and
$\{Q_s\}_{s=1}^\infty$ respectively, are families of noncommutative
polynomials in $\Bbb C\langle X_1,\ldots, X_n, Y_1,\ldots,
Y_m\rangle$, and $\Bbb C\langle X_1,\ldots, X_n \rangle$
respectively, with rational coefficients.

Given $R> \max_{1\leq j\leq
n}\Vert x_{j}\Vert+ \max_{1\leq j\leq p}\Vert y_{j}\Vert $, $r,s\in \mathbb{N}$ and $\epsilon >0$, we can find $%
r_{1},s_1\in \mathbb{N}$ and $\epsilon _{1}>0, \epsilon_2>0$ such that, for all $k\in \mathbb{N%
}$,$$
\begin{aligned}
\Gamma _{R}^{(top)}(x_{1},\ldots ,x_{n}; k,\epsilon_1, Q_1,\ldots,
Q_{s_1})&\subseteq \ \Gamma_{R}^{(top)}(x_{1},\ldots
,x_{n}:y_{1},\ldots ,y_{p}; k,\epsilon, P_1,\ldots,
P_r)\\
 \Gamma_{R}^{(top)}(x_{1},\ldots
,x_{n}:y_{1},\ldots ,y_{p}; k,\epsilon_2, P_1,\ldots, P_{r_1}) &
\subseteq \ \Gamma _{R}^{(top)}(x_{1},\ldots ,x_{n}; k,\epsilon,
Q_1,\ldots, Q_s ).
\end{aligned}$$
 Hence
\[
\begin{aligned}
o_2&(\Gamma_R^{(top)}(x_1,\ldots,x_n; k,\epsilon_1, Q_1,\ldots,
Q_{s_1}),4\omega)  \le o_2( \Gamma_R^{(top)}(x_1,\ldots,x_n:
y_1,\ldots,y_p;k,\epsilon, P_1,\ldots,
P_r),2\omega)\\
o_2&( \Gamma_R^{(top)}(x_1,\ldots,x_n:
y_1,\ldots,y_p;k,\epsilon_2, P_1,\ldots, P_{r_1}),2\omega)    \le
o_2(\Gamma_R^{(top)}(x_1,\ldots,x_n;k,\epsilon, Q_1,\ldots, Q_s
),\omega),
\end{aligned}
\]%
for all $ \omega>0 $. Therefore, for all $ \omega>0 $,
\[
\begin{aligned}
\inf_{\epsilon_1>0, s_1\in\Bbb N} \ \limsup_{k\rightarrow\infty}
&\frac {\log(o_2 (\Gamma_R^{(top)}(x_1,\ldots,x_n; k,\epsilon_1,
Q_1,\ldots, Q_{s_1}),4\omega))}{k^2} \\& \le
\limsup_{k\rightarrow\infty} \frac {\log(  o_2(
\Gamma_R^{(top)}(x_1,\ldots,x_n: y_1,\ldots,y_p;k,\epsilon,
P_1,\ldots, P_r),2\omega))}{k^2};\end{aligned}
\] and
\[
\begin{aligned}
\inf_{\epsilon_2>0, r_1\in\Bbb N} \ \limsup_{k\rightarrow\infty}
&\frac {\log(o_2 ( \Gamma_R^{(top)}(x_1,\ldots,x_n:
y_1,\ldots,y_p;k,\epsilon_2, P_1,\ldots, P_{r_1}),2\omega))}{k^2}
\\& \le \limsup_{k\rightarrow\infty} \frac {\log(o_2(\Gamma_R^{(top)}(x_1,\ldots,x_n;k,\epsilon, Q_1,\ldots, Q_s
),\omega))}{k^2}.\end{aligned}
\]
It follows that, for all $ \omega>0 $,
\[
\begin{aligned}
\inf_{\epsilon_1>0, s_1\in\Bbb N} \ \limsup_{k\rightarrow\infty}
&\frac {\log(o_2 (\Gamma_R^{(top)}(x_1,\ldots,x_n; k,\epsilon_1,
Q_1,\ldots, Q_{s_1}),4\omega))}{k^2} \\& \le \inf_{\epsilon>0,
r\in\Bbb N}  \ \limsup_{k\rightarrow\infty} \frac {\log(  o_2(
\Gamma_R^{(top)}(x_1,\ldots,x_n: y_1,\ldots,y_p;k,\epsilon,
P_1,\ldots, P_r),2\omega))}{k^2};\end{aligned}
\] and
\[
\begin{aligned}
\inf_{\epsilon_2>0, r_1\in\Bbb N} \ \limsup_{k\rightarrow\infty}
&\frac {\log(o_2 ( \Gamma_R^{(top)}(x_1,\ldots,x_n:
y_1,\ldots,y_p;k,\epsilon_2, P_1,\ldots, P_{r_1}),2\omega))}{k^2}
\\& \le \inf_{\epsilon>0, s\in\Bbb N}  \ \limsup_{k\rightarrow\infty} \frac {\log(o_2(\Gamma_R^{(top)}(x_1,\ldots,x_n;k,\epsilon, Q_1,\ldots, Q_s
),\omega))}{k^2}.\end{aligned}
\]
The rest follows from the definitions.
\end{proof}

\subsection{ }  Our next result shows that the topological orbit dimension
$\frak K_{top}^{(2)}$ is in fact a C$^*$-algebra invariant. 

\begin{theorem}
Suppose that $\mathcal A$ is a unital C$^*$-algebra and
$\{x_1,\ldots, x_n\}$, $\{y_1,\ldots, y_p\}$ are two families of
self-adjoint generators of $\mathcal A$. Then
$$
\frak K_{top}^{(2)}(x_1,\ldots, x_n) = \frak
K_{top}^{(2)}(y_1,\ldots, y_p).
$$
\end{theorem}

\begin{proof}  \textbf{\ } Note $x_{1},\ldots,x_{n}$ are elements in
$\mathcal{A}$ that generate $\mathcal{A}$ as a C$^*$-algebra. For
every $0<\omega<1$,
there exists a family of noncommutative polynomials $\psi_{i}(x_{1}%
,\ldots,x_{n})$, $1\leq i\leq p$, such that
\[
\left(\sum_{i=1}^{p}\Vert
y_{i}-\psi_{i}(x_{1},\ldots,x_{n})\Vert^2\right)^{1/2}<
\frac{\omega}{4} .
\]
For such a family of polynomials $\psi_{1},\ldots,\psi_{p}$, and
every $R>\max\{\|x_1\|,\ldots, \|x_n\|,\|y_1\|,\ldots, \|y_p\|\}$
there always exists a constant $D\geq1$, depending only on $R,\psi_{1}%
,\ldots,\psi_{n}$, such that
\[
\left (
\sum_{i=1}^{p}\Vert\psi_{i}(A_{1},\ldots,A_{n})-\psi_{i}(B_{1},\ldots
,B_{n})\Vert_{2}^{2}\right )^{1/2}\leq D
\Vert(A_{1},\ldots,A_{n})-(B_{1},\ldots ,B_{n})\Vert_{2},
\]
for all $(A_{1},\ldots,A_{n}),(B_{1},\ldots,B_{n})$ in $\mathcal{M}%
_{k}(\mathbb{C})^{ {n}}$, all $k\in\mathbb{N}$, satisfying $\Vert
A_{j}\Vert, \Vert B_{j}\Vert\leq R,$ for $1\leq j\leq n.$

Suppose that $\{P_r\}_{r=1}^\infty$, and $\{Q_s\}_{s=1}^\infty$
respectively, is the family of noncommutative polynomials in $\Bbb
C\langle Y_1,\ldots, Y_p, X_1,\ldots, X_n \rangle$, and $\Bbb
C\langle X_1,\ldots, X_n \rangle$ respectively, with rational
coefficients.

For any $s\ge 1$, $\epsilon>0$, when $ r$ is sufficiently large,
$\epsilon'$ sufficiently small, every
$$(H_{1},\ldots,H_{p},A_{1},\ldots,A_{n})$$ in
$$\Gamma_{R}^{(top)}(y_{1},\ldots,y_{p},x_{1},\ldots,x_{n};k,\epsilon', P_1,\ldots, P_r)$$
satisfies
\[
\left( \sum_{i=1}^{p}\Vert
H_{i}-\psi_{i}(A_{1},\ldots,A_{n})\Vert^2\right)^{1/2} \leq
\frac{\omega}{4} ;
\]
and
$$(A_{1},\ldots,A_{n})\in \Gamma_{R}^{(top)}( x_{1},\ldots,x_{n};k,\epsilon,
Q_1,\ldots, Q_s).$$  On the other hand, by the definition of the
orbit covering number, we know there exists a set $\{
\mathcal{U}(B_{1}^{\lambda
},\ldots,B_{n}^{\lambda};\frac{\omega}{4D},\| \cdot \|_2)
\}_{\lambda\in\Lambda_{k}}$ of
$\frac{\omega}{4D}$-orbit-$\|\cdot\|_2$-balls that cover
$\Gamma_{R}^{(top)}( x_{1},\ldots,x_{n};k,\epsilon , Q_1,\ldots,
Q_s)$ with the cardinality of $ \Lambda_{k} $ satisfying \
$|\Lambda_{k}|=o_2(\Gamma_{R}^{(top)}(
x_{1},\ldots,x_{n};k,\epsilon,
Q_1,\ldots, Q_s),\frac{\omega}{4D}).$  Thus   for   such   $(A_{1}%
,\ldots,A_{n})$ in $\Gamma_{R}^{(top)}(
x_{1},\ldots,x_{n};k,\epsilon, Q_1,\ldots, Q_s)$, there exists some
$\lambda\in\Lambda_{k}$ and $W\in\mathcal{U}(k)$ such that
\[
\Vert(A_{1},\ldots,A_{n})-(WB_{1}^{\lambda}W^{\ast},\ldots,WB_{n}^{\lambda
}W^{\ast})\Vert_{2}\leq\frac{\omega}{4D}.
\]
It follows that
\[ \begin{aligned}
&\left(\sum_{i=1}^{p} \|
H_{i}-W\psi_{i}(B_{1}^{\lambda},\ldots,B_{n}^{\lambda
})W^{\ast}\Vert_{2}^{2}\right)^{1/2} =\left(\sum_{i=1}^{p}\Vert
H_{i}-\psi_{i}(WB_{1}^{\lambda
}W^{\ast},\ldots,WB_{n}^{\lambda}W^{\ast})\Vert_{2}^{2}\right)^{1/2}\\
&\qquad \le\left(\sum_{i=1}^{p}\|
H_{i}-\psi_i(A_{1},\ldots,A_{n})\|_2^2\right)^{1/2}+\left(\sum_{i=1}^{p}\|\psi_i(A_{1},\ldots,A_{n})-\psi_{i}
(WB_{1}^{\lambda}W^{\ast},\ldots,WB_{n}^{\lambda}W^{\ast})\Vert_{2}^{2}\right)^{1/2}\\
&\qquad \le \left (\sum_{i=1}^{p}\Vert
H_{i}-\psi_i(A_{1},\ldots,A_{n})\|^2\right)^{1/2}+
\frac{\omega}{4}\\& \qquad \leq  \frac{\omega}{2} ,\end{aligned}
\]
for some $\lambda\in\Lambda_{k}$ and $W\in\mathcal U(k),$ i.e.,
\[
(H_{1},\ldots,H_{p})\in\mathcal{U}(\psi_{1}(B_{1}^{\lambda},\ldots
,B_{n}^{\lambda}),\ldots,\psi_{p}(B_{1}^{\lambda},\ldots,B_{n}^{\lambda
});\frac \omega 2).
\]
Hence, for given $s\in \Bbb N$ and $\epsilon>0$, when $\epsilon'$ is
small enough and $r$ is large enough,
$$\begin{aligned}
o_2&(\Gamma_{R}^{(top)}(y_1,\ldots, y_p:
    x_{1},\ldots,x_{n};k,\epsilon' ,
P_1,\ldots, P_r),\omega)\\
&\qquad  \le |\Lambda_k|=o_2(\Gamma_{R}^{(top)}(
x_{1},\ldots,x_{n};k,\epsilon, Q_1,\ldots, Q_s),\frac{\omega}{4D}).
\end{aligned}
$$It follows that
$$
\begin{aligned}
 \inf_{r\in\Bbb N , \ \epsilon'>0}  \lim_{k\rightarrow \infty}& \frac{\log(o_2(\Gamma_{R}^{(top)}(y_1,\ldots, y_p:
    x_{1},\ldots,x_{n};k,\epsilon' ,
P_1,\ldots, P_r),\omega))}{k^2}\\
& \quad \le \lim_{k\rightarrow \infty}
\frac{\log(o_2(\Gamma_{R}^{(top)}( x_{1},\ldots,x_{n};k,\epsilon,
Q_1,\ldots, Q_s),\frac{\omega}{4D}))}{k^2}.
\end{aligned}
$$
Therefore by the definition of the topological orbit dimension and
Remark 2.5, we get
\[
\begin{aligned}
\frak K_{top}^{(2)}&(y_1,\ldots,y_p: x_1,\ldots,x_n;\omega)=
  \frak K_{top}^{(2)} (y_1,\ldots,y_p: x_1,\ldots,x_n;\omega,R)\\
  &=
   \inf_{\epsilon'>0, r\in \Bbb N} \lim_{k\rightarrow \infty}\frac {\log(o_2(\Gamma_{R}^{(top)}(y_1,\ldots, y_p:
    x_{1},\ldots,x_{n};k,\epsilon' ,
P_1,\ldots, P_r),\omega))}{k^2}
   \\
&\le\inf_{\epsilon>0, s\in \Bbb N} \limsup_{k\rightarrow \infty}
\frac {\log(o_2(\Gamma_{R}^{(top)}( x_{1},\ldots,x_{n};k,\epsilon,
Q_1,\ldots, Q_s),\frac
\omega{4D}))}{k^2 }\\
&\le \frak K_{top}^{(2)} (  x_1,\ldots,x_n ),
\end{aligned}
\]where the last inequality follows from the fact that $\frak K_{top}^{(2)} (  x_1,\ldots,x_n
;\omega)$ increases as $\omega$ decreases. Thus, by Lemma 3.2, we
get
$$
\frak K_{top}^{(2)}(y_1,\ldots, y_p)\le \frak K_{top}^{(2)} (
x_1,\ldots,x_n ).
$$ Similarly, we have $$
 \frak K_{top}^{(2)} ( x_1,\ldots,x_n ) \le \frak
K_{top}^{(2)}(y_1,\ldots, y_p),
$$  which completes the proof.
\end{proof}

A slight modification of the proof of Theorem 3.2 will show the
following result.
\begin{theorem}
Suppose that $\mathcal A$ is a unital C$^*$-algebra with a family of
self-adjoint generators $y_1,\ldots, y_n$. Suppose $\mathcal A_i$,
$j=1,2,\ldots$ is an increasing sequence of unital C$^*$-subalgebras
of $\mathcal A$ such that $\cup_{j=1}^\infty \mathcal A_j$ is norm
dense in $\mathcal A$. Suppose $x_1^{(j)},\ldots, x_{n_j}^{(j)}$ is
a  family of self-adjoint generators of $\mathcal A_j$ for
$j=1,2,\ldots$. Then
$$
\frak K_{top}^{(2)}(y_1,\ldots, y_p) \le \liminf_{j\rightarrow
\infty} \frak K_{top}^{(2)}(x_1^{(j)},\ldots, x_{n_j}^{(j)}).
$$
\end{theorem}

\begin{proof}
\textbf{\ } Note $\cup_{j=1}^\infty \mathcal A_i$ is norm dense in
$\mathcal A$ and, for each $j\ge 1$, $x_1^{(j)},\ldots,
x_{n_j}^{(j)}$ is a family of self-adjoint generators of $\mathcal
A_j$. For every $0<\omega<1$,
there exist a positive integer $j$ and  a family of noncommutative polynomials $\psi_{i}(x_{1}^{(j)}%
,\ldots,x_{n_j}^{(j)})$, $1\leq i\leq p$, such that
\[
\left (\sum_{i=1}^{p}\Vert
y_{i}-\psi_{i}(x_{1}^{(j)},\ldots,x_{n_j}^{(j)})\Vert^2\right)^{1/2}<
\frac{\omega}{4} .
\] The rest of the proof is identical to the one of Theorem 3.2.
\end{proof}

\subsection{ } Suppose $\mathcal A$ is a finitely generated
unital C$^*$-algebra and $n$ is a positive integer. In this
subsection, we are going to compute the topological orbit dimension
in the unital C$^*$-algebra $\mathcal A\otimes \mathcal M_n(\Bbb
C)$.

Assume that $\{e_{st}\}_{s,t=1}^n$ is a canonical system of matrix
units in $\mathcal M_n(\Bbb C)$ and $I_n$ is the identity matrix of
$\mathcal M_n(\Bbb C)$.

The following statement is an easy adaption of  Lemma 2.3 in
\cite{BR}.
\begin{lemma} [Lemma 2.3 in \cite{BR}]
  For any $\epsilon>0$, there is a   constant $\delta>0$ so that the following
  holds: For any $k\in \Bbb N$, if $\{E_{st}\}_{s,t=1}^n$ is a family of elements in
  $\mathcal M_k(\Bbb C)$ satisfying:
  $$
   \| E_{s_1t_1}E_{s_2t_2}-\delta_{ t_1s_2}E_{s_1t_2}\|\le \delta, \  \|
   E_{s_1t_1}-E_{t_1s_1}^*\|\le \delta, \ \|\sum_{i=1}^n
   E_{ii}-I_k\|\le \delta,  \ \forall 1\le s_1,s_2,t_1,t_2\le n
  $$ (where $\delta_{t_1s_2}$ is $1$ if $t_1=s_2$ and is $0$ if $t_1\ne
  s_2$),  then (i) $n|k$; and (ii) there is some unitary matrix $W$ in $\mathcal
  M_k(\Bbb C)$ such that
  $$
\sum_{s,t=1}^n \|W^*E_{st}W- I_{k/n}\otimes e_{st} \| \le  \epsilon.
 $$
\end{lemma}

\begin{lemma}
Suppose $\mathcal A$ is a unital C$^*$-algebra generated by a family
of self-adjoint elements $x_1,\ldots, x_m$. Suppose that
$\{P_r\}_{r=1}^\infty$, and $\{Q_s\}_{s=1}^\infty$ respectively, is
the family of noncommutative polynomials in $\Bbb C\langle
X_1,\ldots, X_m, \{Y_{st}\}_{s,t=1}^n \rangle$, and $\Bbb C\langle
X_1,\ldots, X_m \rangle$ respectively, with rational coefficients.

Let $R> \max\{\|x_1\|, \ldots, \|x_m\|, 1\}$. For any $\omega>0$,
$r_0>0$ and $\epsilon_0>0$, there are some $r>0$ and $\epsilon>0$
such that the following holds:
 if
 $$
(A_1,\ldots, A_m, \{E_{st}\}_{s,t=1}^n) \in
\Gamma_{R}^{(top)}(x_1\otimes I_n,\ldots, x_m\otimes
I_n,\{I_{\mathcal A}\otimes e_{st}\}_{s,t=1}^n, k,\epsilon,
P_1,\ldots, P_{r})\ne \emptyset,
 $$ then (i) $n| k$; (ii) there are a unitary matrix $W$
 in $\mathcal M_k(\Bbb C)$ and
$$
  (B_1, \ldots, B_m) \in \Gamma_{R}^{(top)}(x_1,\ldots, x_m, \frac k{n},\epsilon_0,
Q_1,\ldots, Q_{r_0}).
 $$
  such that
 $$
\sum_{s,t=1}^n \|W^*E_{st}W- I_{k/n}\otimes e_{st} \| +\sum_{i=1}^m
 \|  W^*A_iW  -B_i\otimes I_{n}\|\le
\omega  .
 $$

\end{lemma}

\begin{proof}
We will prove the result by using contradiction. Suppose, to the
contrary, that the result of the lemma does not hold. There are
$\omega>0$, $r_0\in \Bbb N$, and $\epsilon_0>0$ such that, for any
$r\in \Bbb N$,  there are $k_r\in \Bbb N$ and
 $$
(A_1^{(r)},\ldots, A_m^{(r)}, \{E_{st}^{(r)}\}_{s,t=1}^n) \in
\Gamma_{R}^{(top)}(x_1\otimes I_n,\ldots, x_m\otimes
I_n,\{I_{\mathcal A}\otimes e_{st}\}_{s,t=1}^n, k_r,1/r, P_1,\ldots,
P_{r})\ne \emptyset,
 $$ satisfying either $n\nmid k_r$, or if $W$ is a unitary matrix   in $\mathcal M_{k_r}(\Bbb
 C)$ and
 \begin{equation}
  (B_1, \ldots, B_m) \in \Gamma_{R}^{(top)}(x_1,\ldots, x_m, \frac k{n},\epsilon_0,
Q_1,\ldots, Q_{r_0}). \end{equation} then
 \begin{equation}
\sum_{s,t=1}^n \|W^*E_{st}^{(r)}W- I_{k/n}\otimes e_{st} \|
+\sum_{i=1}^m
 \|  W^*A_i^{(r)}W  -B_i\otimes I_{n}\|>
\omega  .\end{equation}

Let $\gamma$ be a free  ultra-filter in $\beta(\Bbb N)\setminus
\Bbb N$. Let $\prod_{r=1}^\gamma \mathcal M_{k_r}(\Bbb C)$ be the
C$^*$ algebra ultra-product of matrices algebras $(\mathcal
M_{k_r}(\Bbb C))_{r=1}^\infty$ along the ultra-filter $\gamma$,
i.e. $\prod_{r=1}^\gamma \mathcal M_{k_r}(\Bbb C)$ is the quotient
algebra of the C$^*$-algebra $\prod_r \mathcal M_{k_r}(\Bbb C)$ by
$\mathcal I_\infty$, where $\mathcal
I_\infty=\{(Y_r)_{r=1}^\infty \in \prod_r \mathcal M_{k_r}(\Bbb
C) \ |   \lim_{r\rightarrow \gamma} \|Y_r\|=0\} $.

Let $\psi$ be the $*$-isomorphism from the C$^*$-algebra $\mathcal
A\otimes \mathcal M_n(\Bbb C)$ into the C$^*$-algebra
$\prod_{r=1}^\gamma \mathcal M_{k_r}(\Bbb C)$ induced by the mapping
$$
   x_i\otimes I_n \rightarrow [(A_i^{(r)})_r]\in \prod_{r=1}^\gamma \mathcal M_{k_r}(\Bbb C),
    \ I_{\mathcal A}\otimes e_{st} \rightarrow [(E_{st}^{(r)})_r]\in \prod_{r=1}^\gamma \mathcal M_{k_r}(\Bbb C)
    \qquad \forall \ 1\le i\le m, 1\le s,t\le n.
$$
Thus $\{\psi(I_{\mathcal A}\otimes e_{st})\}_{s,t=1}^n$ is also a
system of matrix units of a C$^*$-subalgebra ($*$-isomorphic to
$\mathcal M_n(\Bbb C)$) in $\prod_{r=1}^\gamma \mathcal M_{k_r}(\Bbb
C)$. By the preceding lemma, without loss of generality, we can
assume that $n|k_r$ and there is a sequence of unitary matrices
$\{W_r\}_{r=1}^\infty$ where $W_r$ is in $\mathcal M_{k_r}(\Bbb C)$
such that
\begin{equation}
 [(E_{st}^{(r)})_r] = [(W_r (I_{k_r/n}\otimes e_{st})W_r^*)_r], \qquad
 \forall \ 1\le s,t\le n.
\end{equation}
Note that
$$
[(A_i^{(r)})_r][(E_{st}^{(r)})_r] =[(E_{st}^{(r)})_r][(A_i^{(r)})_r]
, \qquad \forall \ 1\le i\le m, 1\le s,t\le n.
$$
Thus by (3.3), there are $B_1^{(r)}, \ldots ,B_m^{(r)}$ in $\mathcal
M_{k_r/n}(\Bbb C)$ for each $r\ge 1$ such that
$$
[(A_i^{(r)})_r]=[(W_r (B_i^{(r)}\otimes I_n) W_r^*)_r], \qquad
\forall \ 1\le i\le m,$$
 which contradicts with our assumptions (3.1), (3.2) and (3.3). This
 completes the proof of the lemma.

\end{proof}

Now we are ready to prove the main result in this subsection.
\begin{theorem}
Suppose that $\mathcal A$ is a unital C$^*$-algebra and $n$ is a
positive integer. Suppose that $\mathcal B= \mathcal A\otimes
\mathcal M_n(\Bbb C)$. If $x_1,\ldots, x_m$ is a family of
self-adjoint generators of $\mathcal A$ and $y_1,\ldots, y_p$ is a
family of self-adjoint generators of $\mathcal B$, then
$$
  \frak K_{top}^{(2)}(y_1,\ldots, y_p) \le \frak
  K_{top}^{(2)}(x_1,\ldots, x_m).
$$
\end{theorem}

\begin{proof} Suppose that
$\{P_r\}_{r=1}^\infty$, and $\{Q_s\}_{s=1}^\infty$ respectively, is
a family of noncommutative polynomials in $\Bbb C\langle X_1,\ldots,
X_m, \{Y_{st}\}_{s,t=1}^n \rangle$, and $\Bbb C\langle X_1,\ldots,
X_m \rangle$ respectively, with rational coefficients.

Let $R> \max\{\|x_1\|, \ldots, \|x_m\|, 1\}$. For any $\omega>0$,
$r_0>0$ and $\epsilon_0>0$, by the preceding lemma, there is a $r>0$
such that, $\forall \ k \in \Bbb N,$
$$\begin{aligned}
o_2(\Gamma_R^{(top)}&(x_1\otimes I_n,\ldots, x_m\otimes I_{n},
\{I_{\mathcal A}\otimes e_{st}\}_{s,t=1}^n; k,1/r, P_1,\ldots, P_r,
2\omega) \\ & \le o_2(\Gamma_R^{(top)}(x_1,\ldots, x_n; k ,1/r_0,
Q_1,\ldots, Q_{r_0}, \omega) .
\end{aligned}
$$
Thus,
$$
   \begin{aligned}
     \inf_{r\in \Bbb N}   \ & \limsup_{k\rightarrow\infty}  \frac {\log(o_2(\Gamma_R^{(top)}(x_1\otimes I_n,\ldots, x_m\otimes I_{n}, \{I_{\mathcal A}\otimes e_{st}\}_{s,t=1}^n; k,1/r, P_1,\ldots, P_r,
2\omega)) }{k^2} \\
& \qquad \le \limsup_{k\rightarrow\infty}  \frac
{\log(o_2(\Gamma_R^{(top)}(x_1,\ldots, x_n; k ,1/r_0, Q_1,\ldots,
Q_{r_0},  \omega)) }{k^2}.
   \end{aligned}
$$ So
$$
   \begin{aligned}
     \inf_{r\in \Bbb N}  \ & \limsup_{k\rightarrow\infty}  \frac {\log(o_2(\Gamma_R^{(top)}(x_1\otimes I_n,\ldots, x_m\otimes I_{n}, \{I_{\mathcal A}\otimes e_{st}\}_{s,t=1}^n; k,1/r, P_1,\ldots, P_r,
2\omega)) }{k^2} \\
& \qquad \le   \inf_{r_0\in \Bbb N} \ \limsup_{k\rightarrow\infty}
\frac {\log(o_2(\Gamma_R^{(top)}(x_1,\ldots, x_n; k,1/r_0,
Q_1,\ldots, Q_{r_0},  \omega)) }{k^2 }.
   \end{aligned}
$$
It follows easily that
$$
  \frak K_{top}^{(2)}(x_1\otimes I_n,\ldots, x_m\otimes I_{n}, \{I_{\mathcal A}\otimes e_{st}\}_{s,t=1}^n ) \le \frak
  K_{top}^{(2)}(x_1,\ldots, x_m).
$$
By Theorem 3.2, we have
$$
  \frak K_{top}^{(2)}(y_1,\ldots, y_p) \le \frak
  K_{top}^{(2)}(x_1,\ldots, x_m),
$$ where $y_1,\ldots, y_p$ is a
family of self-adjoint generators of $\mathcal B$
\end{proof}

The following corollary follows directly from the preceding
theorem.

\begin{corollary}
Suppose that $\mathcal A$ is a unital C$^*$-algebra with a family
of self-adjoint generators $x_1,\ldots, x_m$. Suppose that $n$ is
a positive integer and $\mathcal B=\mathcal A\otimes \mathcal
M_n(\Bbb C)$. If
$$
 \frak
  K_{top}^{(2)}(x_1,\ldots, x_m)=0,
$$ then
$$
  \frak K_{top}^{(2)}(y_1,\ldots, y_p)=0\qquad and \qquad
  \delta_{top}(y_1,\ldots, y_p)\le 1, $$ where $y_1,\ldots, y_p$ is  any family of
  self-adjoint generators of $\mathcal B$.
\end{corollary}

\begin{example}
Suppose that $x_1,\ldots, x_m$ is a  family of self-adjoint
generators of a full matrix algebra $\mathcal M_n(\Bbb C)$. Then
$$
 \frak
  K_{top}^{(2)}(x_1,\ldots, x_m)=0.
$$
\end{example}
\subsection{ }

In this subsection, we assume that $\mathcal A$ and $\mathcal B$ are
two unital C$^*$-algebras and $\mathcal A \bigoplus \mathcal B$ is
the orthogonal sum of $\mathcal A$ and $\mathcal B$.  We assume
$x_1,\ldots, x_n$, or $y_1,\ldots, y_m$, is a family of self-adjoint
generators of $\mathcal A$, or $\mathcal B$ respectively.  Suppose
that $\{P_r\}_{r=1}^\infty$, and $\{Q_s\}_{s=1}^\infty$
respectively, is the family of noncommutative polynomials in $\Bbb
C\langle X_1,\ldots, X_n  \rangle$, and $\Bbb C\langle Y_1,\ldots,
Y_m \rangle$ respectively, with rational coefficients. Suppose that
$\{S_r\}_{r=1}^\infty$  is the family of noncommutative polynomials
in $\Bbb C\langle X_1,\ldots, X_n,
 Y_1,\ldots, X_m
\rangle$   with rational coefficients.

Let $R>\max\{\|x_1\|, \ldots,\|x_n\|, \|y_1\|,\ldots,\|y_m\|\}$ be a
positive number. By the definition of topological orbit dimension,
we have the following.
\begin{lemma}
Let
$$
\alpha>\frak K_{top}^{(2)}(x_1,\ldots,x_n) \qquad \text { and }
\qquad \beta>\frak K_{top}^{(2)}(y_1,\ldots,y_m).
$$ (i)  For each $ \omega> 0$, there is $r(\omega)$
satisfying
$$
   \begin{aligned}
       \limsup_{k_1\rightarrow \infty} & \frac {\log(o_2(\Gamma_R^{(top)}(x_1,\ldots,x_n;k_1,\frac 1 {r(\omega)},
      P_1,\ldots,P_{r(\omega)}),\omega))}{ k_1^2 } <\alpha;\\
        \limsup_{k_2\rightarrow \infty} & \frac {\log(o_2(\Gamma_R^{(top)}(y_1,\ldots,y_m;k_2,\frac 1 {r(\omega)},
      Q_1,\ldots,Q_{r(\omega)}),\omega))}{ k_2^2 } <\beta.
   \end{aligned}
$$
(ii) Therefore, for each $ \omega> 0$ and $r(\omega)\in \Bbb N$,
there is some $K(r(\omega))\in \Bbb N$ satisfying
$$
   \begin{aligned}
       & {\log(o_2(\Gamma_R^{(top)}(x_1,\ldots,x_n;k_1,\frac 1 {r(\omega)},
      P_1,\ldots,P_{r(\omega)}),\omega))} <\alpha{ k_1^2 }, \quad \forall \ k_1\ge K( r(\omega));\\
         &   {\log(o_2(\Gamma_R^{(top)}(y_1,\ldots,y_m;k_2,\frac 1 {r(\omega)},
      Q_1,\ldots,Q_{r(\omega)}),\omega))} < \beta{k_2^2 }, \quad \forall \ k_2\ge
      K( r(\omega)).
   \end{aligned}
$$
\end{lemma}

\begin{lemma}
Suppose that $\mathcal A$ and $\mathcal B$ are two unital C$^*$
algebras and $x_1 , \ldots, x_n$, or $y_1, \ldots,  y_m$ is a family
of self-adjoint elements that generates $\mathcal A$, or $ \mathcal
B$ respectively.

Let $R>\max\{\|x_1\|, \ldots,\|x_n\|,  \|y_1\|,\ldots,\|y_m\|\}$ be
a positive number. For any $\omega>0$, $r_0\in \Bbb N$, there is
some $t>0$ so that the following holds: $\forall \ r>t$, $\forall \
k\ge 1$, if
$$
  (X_1,\ldots, X_n, Y_1,\ldots, Y_m) \in \Gamma_R^{(top)}(x_1\oplus 0 ,  \ldots,x_n \oplus  0 ,  0 \oplus y_1, \ldots,
    0 \oplus y_m;k,\frac 1 r,
      S_1,\ldots,S_r),
$$ then there are
$$
   \begin{aligned}
       (A_1,\ldots, A_n) &\in \Gamma_R^{(top)}(x_1  ,\ldots,x_n ;k_1,\frac 1 {r_0},
      P_1,\ldots,P_{r_0}),\\
        (B_1,\ldots, B_m) &\in \Gamma_R^{(top)}(y_1  ,\ldots,y_m ;k_2,\frac 1 {r_0},
      Q_1,\ldots,Q_{r_0})
   \end{aligned}
$$ and $U\in \mathcal U(k)$ so that (i) $k_1+k_2=k$; and (ii)
$$
  \left \| (X_1,\ldots, X_n, Y_1,\ldots, Y_m) -U^* ( A_1\oplus  0 ,\ldots, A_n\oplus 0 ,  0 \oplus  B_1, \ldots,  0 \oplus B_m     )   U        \right
  \|< \omega.
$$

\end{lemma}
\begin{proof}The proof of this lemma is a slight modification of the one of Lemma 4.2 in
\cite{HaSh2}.

\end{proof}
\begin{theorem}
Suppose that $\mathcal A$ and $\mathcal B$ are unital C$^*$-algebras
with a family of self-adjoint generators $x_1,\ldots, x_n$, or
$y_1,\ldots, y_m$ respectively. Suppose $\mathcal D= \mathcal
A\bigoplus \mathcal B$ is the orthogonal sum of $\mathcal A$ and
$\mathcal B$ with a family of self-adjoint generators $z_1,\ldots,
z_p$. Then,
$$
\frak K_{top}^{(2)}(z_1,\ldots, z_p) \le \frak
K_{top}^{(2)}(x_1,\ldots, x_n)+ \frak K_{top}^{(2)}(y_1,\ldots,
y_m).
$$

\end{theorem}
\begin{proof}
Recall that $\{S_r\}_{r=1}^\infty$, $\{P_r\}_{r=1}^\infty$, and
$\{Q_s\}_{s=1}^\infty$ respectively, are the families of
noncommutative polynomials in $\Bbb C\langle X_1,\ldots, X_n,
 Y_1,\ldots, X_m
\rangle$, $\Bbb C\langle X_1,\ldots, X_n \rangle$, and $\Bbb
C\langle Y_1,\ldots, Y_m \rangle$ respectively,     with rational
coefficients.

Let
$$
\alpha>\frak K_{top}^{(2)}(x_1,\ldots,x_n) \qquad \text { and }
\qquad \beta>\frak K_{top}^{(2)}(y_1,\ldots,y_m),
$$  and $R>\max\{\|x_1\|,\ldots, \|x_n\| y_1\|,\ldots,\| y_m\|\}$ be a
positive number. By definition, the values of  topological orbit
dimension $\frak K_{top}^{(2)}$ can only   be $-\infty$ or  $\ge 0$.
Without loss of generality, we can assume that $\alpha>0$ and
$\beta>0$. By Lemma 3.5, for any $ \omega>0$, there are
$r(\omega)\in \Bbb N$ and $K(r(\omega))\in \Bbb N$ satisfying
   \begin{align}
       & { o_2(\Gamma_R^{(top)}(x_1,\ldots,x_n;k_1,\frac 1 {r(\omega)},
      P_1,\ldots,P_{r(\omega)}),\omega) } < e ^{\alpha{ k_1^2}}, \quad \forall \ k_1\ge K( r(\omega));\\
         &   { o_2(\Gamma_R^{(top)}(y_1,\ldots,y_m;k_2,\frac 1 {r(\omega)},
      Q_1,\ldots,Q_{r(\omega)}),\omega)} <e^{\beta{k_2^2}}, \quad \forall \ k_2\ge
      K( r(\omega)).
   \end{align}
On the other hand, for each $ \omega > 0$ and $r(\omega)\in \Bbb N$,
it follows from Lemma 3.6 that there is some $t\in \Bbb N$ so that
$\forall \ r>t$, $\forall \ k\ge 1$, if
$$
  (X_1,\ldots, X_n, Y_1,\ldots, Y_m) \in \Gamma_R^{(top)}(x_1\oplus 0 ,  \ldots,x_n \oplus  0 ,  0 \oplus y_1, \ldots,
    0 \oplus y_m;k,\frac 1 r,
      S_1,\ldots,S_r),
$$ then there are
$$
   \begin{aligned}
       (A_1,\ldots, A_n) &\in \Gamma_R^{(top)}(x_1  ,\ldots,x_n ;k_1,\frac 1 {r_\omega},
      P_1,\ldots,P_{r_\omega}),\\
        (B_1,\ldots, B_m) &\in \Gamma_R^{(top)}(y_1  ,\ldots,y_m ;k_2,\frac 1 {r_\omega},
      Q_1,\ldots,Q_{r_\omega})
   \end{aligned}
$$ and $U\in \mathcal U(k)$ so that (i) $k_1+k_2=k$; and (ii)
$$
  \left \| (X_1,\ldots, X_n, Y_1,\ldots, Y_m) -U^* ( A_1\oplus  0 ,\ldots, A_n\oplus 0 ,  0 \oplus  B_1, \ldots,  0 \oplus B_m     )   U        \right
  \|< \omega.
$$
It follows that
   \begin{align}
     o_2(\Gamma_R^{(top)}&(x_1\oplus 0 ,  \ldots,x_n \oplus  0 ,  0 \oplus y_1, \ldots,
    0 \oplus y_m;k,\frac 1 r,
      S_1,\ldots,S_r),3\omega)\notag\\
      &\le \sum_{k_1+k_2=k} \left ( o_2(\Gamma_R^{(top)}(x_1  ,\ldots,x_n ;k_1,\frac 1 {r_\omega},
      P_1,\ldots,P_{r_\omega}),\omega) \right .\notag\\
      & \quad \qquad \qquad \qquad \left . \cdot\  o_2(\Gamma_R^{(top)}(y_1  ,\ldots,y_m ;k_2,\frac 1 {r_\omega},
      Q_1,\ldots,Q_{r_\omega}),\omega) \right )\notag\\
      & = \left (
      \sum_{k_1=1}^{K(r(\omega))}+\sum_{k_1=K(r(\omega))+1}^{k-K(r(\omega))-1}+\sum_{k_1=k-K(r(\omega))}^k\right) \left ( o_2(\Gamma_R^{(top)}(x_1  ,\ldots,x_n ;k_1,\frac 1 {r_\omega},
      P_1,\ldots,P_{r_\omega}),\omega) \right .\notag\\
      & \quad \qquad \qquad \qquad \qquad \qquad \qquad   \left . \cdot\ o_2(\Gamma_R^{(top)}(y_1  ,\ldots,y_m ;k_2,\frac 1 {r_\omega},
      Q_1,\ldots,Q_{r_\omega}),\omega) \right )
   \end{align}
Let
$$
   \begin{aligned}
      M_\omega & = \max_{1\le k_1\le K(r(\omega))} o_2(\Gamma_R^{(top)}(x_1  ,\ldots,x_n ;k_1,\frac 1 {r_\omega},
      P_1,\ldots,P_{r_\omega}),\omega)+1,\\
      N_\omega & = \max_{1\le k_2\le K(r(\omega))} o_2( \Gamma_R^{(top)}(y_1  ,\ldots,y_m ;k_2,\frac 1 {r_\omega},
      Q_1,\ldots,Q_{r_\omega}),\omega)+1
   \end{aligned}
$$
By (3.4) and (3.5), we know that
$$
\begin{aligned}
  (3.6) & \le K(r(\omega))M_{\omega} e^{\beta k_2^2}+ K(r(\omega))N_{\omega} e^{\alpha
  k_1^2} + (k-2K(r(\omega))) \cdot (e^{\alpha k_1^2+ \beta k_2^2}+1)\\
  & \le K(r(\omega))M_{\omega} e^{\beta k^2}+ K(r(\omega))N_{\omega} e^{\alpha
  k^2} + 2k \cdot e^{(\alpha +\beta)k^2 }\\ & \le 3k \cdot e^{(\alpha +\beta)k^2
  },
\end{aligned}
$$ when $k$ is large enough. Now it is   not hard   to show that
$$
\frak K_{top}^{(2)}(x_1\oplus 0 ,  \ldots,x_n \oplus  0 ,  0 \oplus
y_1, \ldots,
    0 \oplus y_m)\le \alpha+ \beta.
$$
Thus, by Theorem 3.2, we have
$$
\frak K_{top}^{(2)}(z_1,\ldots, z_p) \le \frak
K_{top}^{(2)}(x_1,\ldots, x_m)+ \frak K_{top}^{(2)}(y_1,\ldots,
y_m),
$$ where $z_1,\ldots, z_p$ is any family of self-adjoint generators of
$\mathcal A\bigoplus \mathcal B$.
\end{proof}

The following corollary follows directly from the preceding
theorem.
\begin{corollary}
Suppose that $\mathcal A$ and $\mathcal B$ are unital C$^*$-algebras
with a family of self-adjoint generators $x_1,\ldots, x_n$, and
$y_1,\ldots, y_m$ respectively. Suppose $\mathcal D= \mathcal
A\bigoplus \mathcal B$ is the orthogonal sum of $\mathcal A$ and
$\mathcal B$ with a family of self-adjoint generators $z_1,\ldots,
z_p$. if
$$
  \frak
K_{top}^{(2)}(x_1,\ldots, x_m)=\frak K_{top}^{(2)}(y_1,\ldots,
y_m)=0,
$$ then
$$
 \frak
K_{top}^{(2)}(z_1,\ldots, z_p)=0 \qquad and \qquad
\delta_{top}(z_1,\ldots, z_p)\le 1.
$$

\end{corollary}
By Example 3.1 and Corollary 3.2, we have the following result.
\begin{example}
Suppose $x_1,\ldots, x_n$ is a family of self-adjoint generators
of a finite dimensional C$^*$-algebra $\mathcal B$. Then
$$
  \frak
K_{top}^{(2)}(x_1,\ldots, x_n)=0.$$
\end{example}
\begin{remark}
The  result in the preceding example will be extended to the
general case of a nuclear C$^*$-algebra in Corollary 4.1.
\end{remark}

\section{Orbit dimension capacity}
  In this section, we are going to define
the concept of ``orbit dimension capacity" of $n$-tuple of
elements in a unital C$^*$-algebra, which is an analogue of
``free dimension capacity" in \cite{Voi}.

\subsection{Modified free orbit dimension in finite von Neumann
algebras}

Let $\mathcal{M}$ be a von Neumann algebra with a  tracial state
$\tau$, and $x_{1},\ldots,x_{n}$ be self-adjoint elements in
$\mathcal{M}$. For any positive $R$ and $\epsilon$, and any $m,k$ in
$\mathbb{N}$, let
$\Gamma_{R}(x_{1},\ldots,x_{n};m,k,\epsilon;\tau)$ be the subset of $\mathcal{M}%
_{k}^{s.a}(\mathbb{C})^{n}$ consisting of all $(A_{1},\ldots,A_{n})$
in $\mathcal{M}_{k}^{s.a}(\mathbb{C})^{n}$ such that  $\max_{1\le
j\le n}\| A_j\|\le R$, and
\[
|\tau_{k}(A_{i_{1}} \cdots A_{i_{q}} )-\tau(x_{i_{1}%
} \cdots x_{i_{q}} )|<\epsilon,
\]
for all $1\leq i_{1},\ldots,i_{q}\leq n$,   and   $1\leq q\leq m$.

For any $\omega>0$, let $o_2(\Gamma_R(x_1,\ldots,
x_n;m,k,\epsilon;\tau),\omega)$ be the minimal number of
$\omega$-orbit-$\|\cdot\|_2$-balls in
$\mathcal{M}_{k}(\mathbb{C})^{n}$ that constitute  a covering of
 $\Gamma_R(x_1,\ldots, x_n;m,k,\epsilon;\tau)$.

Now we define, successively,
\[
\begin{aligned}
\frak K^{(2)}_2  (x_1,,\ldots, x_n; \omega;\tau) &= \sup_{R>0} \
\inf_{m\in \Bbb N, \epsilon>0}\limsup_{k\rightarrow \infty} \frac
{\log(o_2(\Gamma_R(x_1,\ldots,
x_n;m,k,\epsilon;\tau),\omega))}{ k^2 } \\
\frak K^{(2)}_2 (x_1,,\ldots, x_n;\tau ) &=
\limsup_{\omega\rightarrow 0^+}\frak K^{(2)}_2 (x_1,,\ldots, x_n;
\omega;\tau),
\end{aligned}
\]
where  $\mathfrak{K}^{(2)}_2 (x_{1} ,\ldots,x_{n};\tau)$   is called
the {\em modified free orbit-dimension} of \ $x_{1},\ldots,x_{n}$
with respect to the tracial state $\tau$.
\begin{remark}
If  the von Neumann algebra $\mathcal M$ with a tracial state $\tau$
is replaced by a unital C$^*$-algebra $\mathcal A$ with a tracial
state $\tau$, then $\mathfrak{K}^{(2)}_2 (x_{1} ,\ldots,x_{n};\tau)$
is still well-defined.
\end{remark}

From the previous definition, it follows directly our next result.
\begin{lemma}
Suppose $x_1,\ldots, x_n$ is a family of self-adjoint elements in a
von Neumann algebra with a tracial state $\tau$. Let $  \mathfrak{K}
_2 (x_{1} ,\ldots,x_{n};\tau)$ be the upper orbit dimension of
$x_{1} ,\ldots,x_{n}$ defined in  Definition 1 of \cite {HaSh}. We
have,   if
$$
\mathfrak{K} _2 (x_{1} ,\ldots,x_{n};\tau)=0,$$ then
$$
\mathfrak{K}^{(2)}_2 (x_{1} ,\ldots,x_{n};\tau)=0.$$
\end{lemma}

\subsection{Definition of orbit dimension capacity} We are are ready
to give the definition of ``orbit dimension capacity".
\begin{definition}
 Suppose that $\mathcal A$ is a unital C$^*$-algebra and $TS(\mathcal
 A)$ is the set of all tracial states of $\mathcal A$. Suppose that $x_1,\ldots, x_n$ is a family of self-adjoint elements in $\mathcal A$. Define
 $$
\frak K\frak K_2^{(2)}(x_1,\ldots, x_n) = \sup_{\tau\in TS(\mathcal
A)}
  \mathfrak{K}^{(2)}_2 (x_{1} ,\ldots,x_{n};\tau)
 $$ to be the orbit dimension capacity of $x_1,\ldots, x_n$.
\end{definition}

\subsection{Topological orbit dimension is   majorized by orbit
dimension capacity} We have the following relationship between
topological orbit dimension and orbit dimension capacity.
\begin{theorem}
Suppose that $\mathcal A$ is a unital C$^*$-algebra and $x_1,\ldots,
x_n$ is a family of self-adjoint elements in $\mathcal A$. Then
$$
\frak K_{top}^{(2)}(x_1,\ldots, x_n) \le \frak K\frak
K_{2}^{(2)}(x_1,\ldots, x_n).
$$
\end{theorem}
 \begin{proof}
The proof is a slight modification of the one in section 3 of
\cite{Voi}. For the sake of the completeness, we also include
 Voiculescu's arguments here.

If $\frak K_{top}^{(2)}(x_1,\ldots, x_n)=-\infty$, there is nothing
to prove. We might assume that $$\frak K_{top}^{(2)}(x_1,\ldots,
x_n)>\alpha>-\infty.  $$ We will show that
$$ \frak K\frak
K_{2}^{(2)}(x_1,\ldots, x_n)=\sup_{\tau\in TS(\mathcal A)} \frak
K_{2}^{(2)}(x_1,\ldots, x_n;\tau)>\alpha.
$$

Let    $\{P_r\}_{r=1}^\infty$be a family of noncommutative
polynomials in  $\Bbb C\langle X_1,\ldots, X_n \rangle$     with
rational coefficients. Let $R>\max\{\|x_1\|,\ldots,\|x_n\|\}$. From
the assumption that $\frak K_{top}^{(2)}(x_1,\ldots, x_n)>\alpha$,
it follows that there exist a positive number $\omega_0>0$ and   a
sequence of positive integers $\{k_{q}\}_{q=1}^\infty$ with
$k_{1}<k_{2}<\cdots$,  so that for some $\alpha'>\alpha$,
$$
 \lim_{q\rightarrow\infty}\frac
{\log(o_{2}(\Gamma^{(top)}_R(x_1,\ldots,x_n;k_{ q}, \frac 1 q,
P_1,\ldots, P_q),\omega_0))}{ k_{ q}^2 }>\alpha'.
$$

 Let
$\mathcal A(n)$ be the universal unital C$^*$-algebra generated by
self-adjoint elements $a_1,\ldots,a_n$ of norm $R$, that is the
unital full free product of $n$ copies of $C[-R,R]$. A microstate
$$\eta=(A_1,\ldots, A_n)\in \Gamma^{(top)}_R(x_1,\ldots,x_n;k_{
q}, \frac1 q,P_1,\ldots, P_q) =\Gamma( q)$$ defines a unital
$*$-homomorphism $\psi_\eta \ : \ \mathcal A(n) \rightarrow \mathcal
M_{k_{ q}}(\Bbb C)$   so that $\psi_\eta(a_i)=A_i$ ($1\le i\le n$)
and a tracial state $\tau_\eta\in TS(\mathcal A(n))$ with
$$
  \tau_\eta = \frac {Tr_{k_{ q}}\circ\psi_{\eta}}{k_{ q}}.
$$
Similarly there is a $*$-homomorphism $\psi \ : \ \mathcal A(n)
\rightarrow \mathcal A$ so that $\psi(a_i)= x_i$, for $1\le i\le n$.

It is not hard to see that the weak topology on $\Omega= TS(\mathcal
A(n))$ is induced by the metric
$$
d(\tau_1,\tau_2) = \sum_{s=1}^\infty \ \sum_{(i_1,\ldots, i_s) \in
(\{1,\ldots, n\})^s} \ (2Rn)^{-s}|(\tau_1-\tau_2)(a_{i_1}\cdots
a_{i_s})|.
$$ Therefore, $\Omega$ is a compact metric space and
$$
 K_{ q}=\{\tau_\eta\in\Omega \ | \ \eta\in \Gamma( q)\}
$$ is a compact subset  of $\Omega$ because $\eta \rightarrow
\tau_\eta$ is continuous and $\Gamma( q)$ is compact. Let further
$K\subseteq \Omega$ denote the compact subset $(TS(\mathcal
A))\circ\psi.$

Given $\epsilon>0$, from the fact that $\Omega$ is compact it
follows that there is some $L(\epsilon)>0$ so that for each $  q\ge
1$,
$$
  K_{ q}=   K_{ q}^1\cup   K_{ q}^2\cup \cdots \cup
  K_{ q}^{L(\epsilon)}
$$
where each compact set $  K_{ q}^{j}$ has diameter $<\epsilon$. Let
$$
\Gamma( q,j) = \{\eta\in\Gamma( q) \ | \ \tau_\eta\in K_{ q}^{j}\}.
$$ We have
$$
   \Gamma( q) = \Gamma( q,1)\cup \cdots \cup
   \Gamma( q,L(\epsilon)).
$$ Let further $\Gamma'( q)$ denote some $\Gamma( q,j)$ such that
$$
  o_2(\Gamma'( q), \omega_0) \ge \frac
  {o_2(\Gamma( q),\omega_0)
  }{L(\epsilon)}.
$$
Thus we have
$$
    \lim_{q\rightarrow\infty} \frac{\log o_2(\Gamma'( q),
   \omega_0)}{ k_{ q}^2 }> \alpha'.
$$

Given $\epsilon$ successively the values $1, 1/2, 1/3, \ldots,
1/s,\ldots$, we can find a subsequence  $\{  q_{ s}\}_{s
=1}^\infty$ such that the chosen set $K_{ q_{ s}}^{j_s}\subseteq
K_{ q_{ s}}$ has diameter $<\frac 1 { s}$ and  the corresponding
set $\Gamma'( q_{ s})$ satisfying
$$
   \lim_{s\rightarrow\infty} \frac{\log o_2(\Gamma'( q_{ s}),
   \omega_{0})}{ k_{ q_{ s}}^2 }> \alpha'.
$$

Without loss of generality,  we can assume that
  $\tau$ is the
weak limit of some sequence $(\tau_{\eta(q_s)})_{s=1}^\infty$. Then
$\tau\in K$. In fact,
$$
    \begin{aligned}
       |\tau(Q(a_1,\ldots, a_n))| & =
       \lim_{s\rightarrow\infty}
       |\tau_{\eta( q_{ s})}(Q(a_1,\ldots, a_n))|\\
       &\le  \limsup_{s\rightarrow\infty}
       ||\psi_{\eta( q_{ s})}(Q(a_1,\ldots, a_n))||\\
       &\le  \lim_{s\rightarrow\infty} (\frac 1 {s}+ \|Q(x_1,\ldots,
       x_n)\|)\\
       &= \|Q(x_1,\ldots,
       x_n)\| \\
       &= \|\psi(Q(a_1,\ldots,
       a_n))\|.
    \end{aligned}
$$ Now it follows from the density of the polynomials $Q$ in
$\mathcal A(n)$ that $\tau\in K$.

We can further assume that there is a subsequence
$\{q_{s(t)}\}_{t=1}^\infty$ of $ \{q_{s }\}_{s=1}^\infty $ so that
the chosen set $K_{q_{s(t)}}^{j_{s(t)}}\subseteq K_{ q_{ s(t)}}$
is $\subseteq B(\tau, 1/t)$, the ball of radius $1/t$ and center
$\tau$.
Therefore,  for any $m\in \Bbb N$ and $\epsilon>0$, we have
$$
   \Gamma'( q_{ s(t)})\subseteq \Gamma_R(x_1,\ldots,x_n; k_{q_{ s(t)}},
   m ,\epsilon ; \tau)
$$ when $t$ is large enough. Thus
$$
\frak K_{2}^{(2)}(x_1,\ldots,x_n;\tau)\ge \frak
K_{2}^{(2)}(x_1,\ldots,x_n;\omega_0; \tau)\ge
\lim_{t\rightarrow\infty} \frac{\log o_2(\Gamma'( q_{ s(t)}),
   \omega_{0})}{ k_{ q_{ s(t)}}^2 }>\alpha'
$$ and hence
$$
 \frak K\frak
K_{2}^{(2)}(x_1,\ldots, x_n)=\sup_{\tau\in TS(\mathcal A)} \frak
K_{2}^{(2)}(x_1,\ldots, x_n;\tau)\ge \frak K_{top}^{(2)}(x_1,\ldots,
x_n).
$$
 \end{proof}


\subsection{Nuclear C$^*$-algebras }

Recall the definition of approximation property of a unital
C$^*$-algebra in \cite{HaSh2} as follows.

\begin{definition}
Suppose $\mathcal A$ is a unital C$^*$ algebra and $x_1,\ldots, x_n$
is a family of self-adjoint elements of $\mathcal A$ that generates
$\mathcal A$ as a C$^*$-algebra. Let    $\{P_r\}_{r=1}^\infty$be a
family of noncommutative polynomials in  $\Bbb C\langle X_1,\ldots,
X_n \rangle$     with rational coefficients. If for any
$R>\max\{\|x_1\|,\ldots,\|x_n\| \}$, $r>0$, $\epsilon>0$, there is a
sequence of positive integers $k_1<k_2<\cdots $ such that
$$
\Gamma^{(top)}_R(x_1,\ldots, x_n ; k_s,\epsilon, P_1,\ldots, P_r)
\ne \varnothing, \qquad \forall \   s\ge 1
$$ then $\mathcal A$ is called having approximation property.

\end{definition}

Now we can compute the topological free entropy dimension of a
family of self-adjoint generators in a unital nuclear C$^*$-algebra.
\begin{corollary}
Suppose $\mathcal A$ is a   unital nuclear C$^*$-algebra with a
family of self-adjoint generators $x_1,\ldots, x_n$. If $\mathcal
A$ has the approximation property in the sense of Definition 4.2,
then
$$
\frak K_{top}^{(2)}(x_1,\ldots, x_n)=0 \qquad \text { and } \qquad
\delta_{top}(x_1,\ldots, x_n) \le 1.
$$
\end{corollary}

\begin{proof}
It is known that every representation of a nuclear C$^*$-algebra
yields an injective von Neumann algebra. From Lemma 4.1 and Theorem
2 in \cite{HaSh}, it follows that
$$
\frak K_{2}^{(2)}(x_1,\ldots, x_n;\tau)=0, \qquad \forall \ \tau\in
TS(\mathcal A),
$$ where $TS(\mathcal A)$ is the set of all tracial states of
$\mathcal A$. Then, by Theorem 4.1 we know that
$$
\frak K_{top}^{(2)}(x_1,\ldots, x_n)=0.$$ By Theorem 3.1,
$$\delta_{top}(x_1,\ldots, x_n) \le 1.
$$
\end{proof}

\begin{remark}
By \cite{BK} (see also Theorem 5.2 in Appendix), we know that a
nuclear C$^*$-algebra $\mathcal A$ has the approximation property in
the sense of Definition 4.2 if and only if $\mathcal A$ is an NF
algebra with a finite family of generators. Thus Corollary 4.1 can
be restated as the following: {\em If $\mathcal A$ is an NF algebra
with a family of self-adjoint generators $x_1,\ldots, x_n$. Then
$$
\frak K_{top}^{(2)}(x_1,\ldots, x_n)=0 \qquad \text { and } \qquad
\delta_{top}(x_1,\ldots, x_n) \le 1.
$$ }
\end{remark}

\subsection{Tensor products }

In this subsection, we are going to prove the following result.

\begin{corollary}
  Suppose that
  $\mathcal B_1$ and $\mathcal B_2$ are two unital C$^*$-algebras
  and $\mathcal A_1$, $\mathcal A_2$ with
  $1\in \mathcal A_1\subseteq \mathcal B_1, 1\in \mathcal A_2\subseteq \mathcal
  B_2$
   are
   infinite dimensional, unital, simple C$^*$-subalgebras with a
  unique tracial state. Suppose that $\mathcal B = \mathcal
  B_1\otimes_\nu \mathcal B_2$ is the C$^*$-tensor product of
  $\mathcal B_1$ and $\mathcal B_2$ with respect to a cross norm
  $\| \cdot \|_\nu$. If $\mathcal B$ has the approximation property
  in the sense of Definition 4.2, then
  $$
\frak K_{top}^{(2)}(x_1,\ldots, x_n)=0\quad \text { and } \quad
\delta_{top}(x_1,\ldots, x_n) = 1,
  $$ where $x_1,\ldots, x_n$ is any family of self-adjoint
  generators of $\mathcal B$.
\end{corollary}

\begin{proof}
Assume that $\tau$ is a tracial state of $\mathcal B$ and $\mathcal
H$ is the Hilbert $L^2(\mathcal B,\tau)$. Let $\psi$ be the GNS
representation of $\mathcal B$ on  $\mathcal H$. Note that $\mathcal
A_1, \mathcal A_2$ are
  infinite dimensional, unital simple C$^*$-algebras with
  a unique tracial state. It is not hard to see that both $\psi(\mathcal
  A_1)$ and $\psi(\mathcal A_2)$ generate diffuse finite von Neumann
  algebras on $\mathcal H$. Thus   both $\psi(\mathcal
  B_1)$ and $\psi(\mathcal B_2)$ generate diffuse finite von Neumann
  algebras on $\mathcal H$.  Moreover, $\psi(\mathcal B_1)$ and $\psi(\mathcal
  B_2)$ commute with each other. Thus, by Corollary 4 in
  \cite{HaSh}, we have that
  $$
    \frak K_2(\psi(x_1), \ldots, \psi(x_n);\tau)=0,
  $$ where $
    \frak K_2(\psi(x_1), \ldots, \psi(x_n);\tau)
  $ is the upper free orbit dimension of $\psi(x_1), \ldots, \psi(x_n)$ (with
  respect to $\tau$) defined in \cite{HaSh}. By Lemma 3.1, we have
$$
    \frak K_2^{(2)}(\psi(x_1), \ldots, \psi(x_n);\tau)=0.
  $$ Since $\tau$ is an arbitrary tracial state of $\mathcal B$, we
  have
$$
  \frak K \frak K_2^{(2)}(x_1, \ldots, x_n)=0.
  $$ By Theorem 3.1, we have
  $$
   \frak K_{top}^{(2)}(x_1, \ldots, x_n)=0.
  $$ Therefore, $$
\delta_{top}(x_1,\ldots, x_n) \le 1.
  $$
 On the other hand, a consequence of Theorem 5.2 in \cite{HaSh2} says that
$$
\delta_{top} (x_1,\ldots,x_n)\ge 1,
$$ if $\mathcal B$ has approximation property in the sense of Definition 4.2. Hence
$$
\frak K_{top}^{(2)}(x_1,\ldots, x_n)=0\quad \text { and } \quad
\delta_{top}(x_1,\ldots, x_n) = 1,
  $$ where $x_1,\ldots, x_n$ is any family of self-adjoint
  generators of $\mathcal B$.
\end{proof}

\begin{example} Assume that $\mathcal A$ is a finite generated
unital C$^*$-algebra with the approximation property in the sense of
Definition 4.2. By Theorem 5.2, Theorem 5.4 in Appendix and
Proposition 3.1 in \cite{HaSh4}, we know that $(C_r^*(F_2)\ast_{\Bbb
C} \mathcal A)\otimes_{min} C_r^*(F_2)$ has the approximation
property, where $C_r^*(F_2)$ is the reduced C$^*$-algebra of free
group $F_2$ . Thus, by Corollary 4.2, we know that, for any family
of self-adjoint generators $x_1,\ldots, x_n$ of
$(C_r^*(F_2)\ast_{\Bbb C} \mathcal A)\otimes_{min} C_r^*(F_2)$,
$$\delta_{top}(x_1,\ldots, x_n) = 1.$$

\end{example}

Similar argument as the preceding corollary shows the following
result.

\begin{corollary}
Suppose that $\mathcal B_1$ is a unital C$^*$-algebra and $\mathcal
B_2$ is an infinite dimension, unital simple C$^*$-algebra with a
unique tracial state $\tau$. Suppose that $\mathcal B = \mathcal
  B_1\otimes_\nu \mathcal B_2$ is the C$^*$-tensor product of
  $\mathcal B_1$ and $\mathcal B_2$ with respect to a cross norm
  $\| \cdot \|_\nu$. Suppose that $z_1,\ldots,z_p$ is a family of
self-adjoint generators of $\mathcal B_2$  and $x_1,\ldots,x_n$ is a
family of self-adjoint generators of $\mathcal B$. If $\mathcal B$
has the approximation property
  in the sense of Definition 4.2 and $\frak K_2^{(2)}(z_1,\ldots,z_p;\tau)=0$, then
  $$
\frak K_{top}^{(2)}(x_1,\ldots, x_n)=0\quad \text { and } \quad
\delta_{top}(x_1,\ldots, x_n) = 1.
  $$

\end{corollary}
\begin{example}
Assume that $\mathcal A$ is a UHF algebra, or an irrational
rotation algebra, or $C_r^*(F_2)\otimes_{min} C_r^*(F_2)$ and
$\mathcal B$ is a finitely generated unital C$^*-$algebra with
the approximation property in the sense of Definition 4.2. Then,
by \cite{BK} or \cite{HaSh4}, $\mathcal A\otimes_{min} \mathcal
B$ has the approximation property. Suppose that $x_1,\ldots, x_n$
is a family of self-adjoint generators of $\mathcal
A\otimes_{min} \mathcal B$. Then
$$\delta_{top}(x_1,\ldots, x_n) = 1.
  $$
\end{example}
\begin{example}
By  Theorem 3.1 of \cite{HaSh4}, $C^*(F_2)\otimes_{max}
(C_r^*(F_2)\otimes_{min} C_r^*(F_2))$ has the approximation
property, where $C^*(F_2)$ is the full C$^*$-algebra of the free
group $F_2$. Suppose that $x_1,\ldots, x_n$ is a family of
self-adjoint generators of $C^*(F_2)\otimes_{max}
(C_r^*(F_2)\otimes_{min} C_r^*(F_2))$. Then
$$\delta_{top}(x_1,\ldots, x_n) = 1.
  $$
\end{example}

\subsection{Crossed products } In this subsection, we are going to prove the
following result.

\begin{corollary}
Suppose that  $ \mathcal A $ is an infinite dimensional unital
simple C$^*$-algebra with a unique tracial state $\tau$. Suppose $G$
is a countable group of   actions $\{\alpha_g\}_{g\in G}$ on
$\mathcal A$. Suppose that $\mathcal D=\mathcal A\rtimes G$ is
either full or reduced crossed product of $\mathcal A$ by the
actions of $G$. Suppose that $z_1,\ldots,z_p$ is a family of
self-adjoint generators of $\mathcal A$  and $x_1,\ldots,x_n$ is a
family of self-adjoint generators of $\mathcal D$. If $\mathcal D$
has the approximation property
  in the sense of Definition 4.2 and $\frak K_2^{(2)}(z_1,\ldots,z_p;\tau)=0$, then
  $$
\frak K_{top}^{(2)}(x_1,\ldots, x_n)=0\quad \text { and } \quad
\delta_{top}(x_1,\ldots, x_n) = 1.
  $$
\end{corollary}

\begin{proof}
Assume that $\tau_1$ is a tracial state of $\mathcal D$ and
$\mathcal H$ is the Hilbert $L^2(\mathcal D,\tau_1)$. Let $\psi$ be
the GNS representation of $\mathcal D$ on  $\mathcal H$. Note that
$\mathcal A$ an infinite dimensional unital simple C$^*$-algebra
with a unique tracial state $\tau$. Thus $\tau_1|_{\mathcal
A}=\tau$. It is not hard to see that $\psi(\mathcal
  A)$   generates a diffuse finite von Neumann algebra on $\mathcal H$. Moreover, for any $g\in G$,
  $$\psi(g^{-1})\psi(\mathcal A) \psi(g)\subseteq \psi(\mathcal A).$$ It follows from the fact that $\frak K_2^{(2)}(z_1,\ldots,z_p;\tau)=0$ and Theorem 4 in
  \cite{HaSh},  that
$$
    \frak K_2^{(2)}(\psi(x_1), \ldots, \psi(x_n);\tau_1)=0.
  $$ Since $\tau_1$ is an arbitrary tracial state of $\mathcal D$, we
  have
$$
  \frak K \frak K_2^{(2)}(x_1, \ldots, x_n)=0.
  $$ By Theorem 3.1, we have
  $$
   \frak K_{top}^{(2)}(x_1, \ldots, x_n)=0.
  $$ Therefore, $$
\delta_{top}(x_1,\ldots, x_n) \le 1.
  $$
 On the other hand,  a consequence of Theorem 5.2 in \cite{HaSh2} says
 that
$$
\delta_{top} (x_1,\ldots,x_n)\ge 1,
$$ if $\mathcal D$ has approximation property in the sense of Definition 4.2. Hence
$$
\frak K_{top}^{(2)}(x_1,\ldots, x_n)=0\quad \text { and } \quad
\delta_{top}(x_1,\ldots, x_n) = 1,
  $$ where $x_1,\ldots, x_n$ is any family of self-adjoint
  generators of $\mathcal D$.
\end{proof}

\begin{example}
Let $C_r^*(F_2)\otimes_{min} C_r^*(F_2)$ be the reduced
C$^*$-algebra of the
 group $F_2\times F_2$.  Let $u_1,u_2$, or
$v_1,v_2$ respectively, be the canonical unitary generators of the
left copy, or the right copy respectively,  of $C_r^*(F_2)$ and
$0<\theta <1$ be a positive number. Let $\alpha$ be a homomorphism
from $\Bbb Z$ into $Aut(C_r^*(F_2)\otimes_{min} C_r^*(F_2))$ induced
by the following mapping: $\forall \ n\in \Bbb Z$, $j=1,2$
$$
\alpha(n)(u_j)=e^{2n \pi    \theta \cdot  i} u_j \qquad and \qquad
\alpha(n)(v_j)=e^{2n\pi   \theta \cdot i} v_j.
$$ Then, by Theorem 4.2 of \cite{HaSh4}, $(C_r^*(F_2)\otimes_{min} C_r^*(F_2))\rtimes_{\alpha}\Bbb Z
$ has the approximation property. Therefore, by Corollary 4.4, we
have
$$\delta_{top}(x_1,\ldots, x_n) = 1,
  $$ where $x_1,\ldots, x_n$ is any family of self-adjoint
  generators of $(C_r^*(F_2)\otimes_{min} C_r^*(F_2))\rtimes_{\alpha}\Bbb Z$.
\end{example}
\section{Topological free entropy dimension in full free products
of unital C$^*$-algebras}

Assume that $\{\mathcal A_i\}_{i=1}^m$ ($m\ge 2$) is a family of
unital C$^*$-algebras. Recall the definition of unital full free
product of $\mathcal A_1,\ldots, \mathcal A_m$ as follows.

\begin{definition}
The unital full free product of the unital C$^*$-algebras
$\{\mathcal A_i\}_{i=1}^m$ ($m\ge 2$) is a unital C$^*$-algebra
$\mathcal D$ equipped with
 unital embedding  $\{ \sigma_i: \ \mathcal A_i\rightarrow \mathcal D\}_{i=1}^m$,  such that (i) the set
$\cup_{i=1}^m \sigma_i(\mathcal A_i) $ is norm dense in $\mathcal
D$; and (ii) if $\phi_i$   is a unital $*$-homomorphism from
$\mathcal A_i$  into a unital C$^*$-algebra $\bar{\mathcal D}$ for
$i=1,2,\ldots, m$, then there is a unital $*$-homomorphisms $\psi $
from $\mathcal D$ to $\bar{\mathcal D}$ satisfying $\phi_i=\psi
\circ \sigma_i$, for $i=1,2,\ldots,m$.
\end{definition}

In this section, for a family of positive integer $n_1,\ldots, n_m$,
we will let $\{X_j^{(i)}\}_{ 1\le i\le m; 1\le j\le n_i}$ be a
family of indeterminates and $\{P_r\}_{r=1}^\infty$ be a family of
noncommutative polynomials in $\Bbb C\langle X_j^{(i)} : 1\le i\le
m; 1\le j\le n_i\rangle$ with rational coefficients. For each  $1\le
i\le m$   and $j\ge 1$, let $ P_j^{(i)} $ be a polynomial in
$X_1^{(i)} , \ldots, X_{n_i}^{(i)}$ defined by:
$$
P_j^{(i)}( X_1^{(i)} , \ldots,  X_{n_i}^{(i)})=P_j(0,\ldots, 0,
X_1^{(i)} , \ldots,  X_{n_i}^{(i)} , 0, \ldots, 0).$$

\begin{lemma}
Suppose $\{\mathcal A_i\}_{i=1}^m$ ($m\ge 2$) is a family of unital
C$^*$-algebras  and   $\mathcal D$  is the unital full free product
of the  C$^*$-algebras  $\{\mathcal A_i\}_{i=1}^m$  equipped with
the unital
  embedding  $\{ \sigma_i:   \mathcal A_i\rightarrow
\mathcal D\}_{i=1}^m$. Suppose $\{x_j^{(i)}\}_{j=1}^{n_i}$ is a
family of self-adjoint generators of $\mathcal A_i$ for $1\le i\le
m$. Then
$$
\delta_{top}(\sigma_1(x_1^{(1)}), \ldots,
\sigma_1(x_{n_1}^{(1)}),\ldots, \sigma_m(x_1^{(m)}),\ldots,
\sigma_m(x_{n_m}^{(m)}))\le \sum_{i=1}^m
\delta_{top}(x_1^{(i)},\ldots, x_{n_i}^{(i)}).
$$
\end{lemma}

\begin{proof}

Let $R>\max_{1\le i\le n, 1\le j\le n_i}\|x_j^{(i)}\| $ be a
positive number. For any $r\in\Bbb N$ and $\epsilon >0$, there is a
positive integer $r_1$ such that
$$\begin{aligned}
\Gamma_R^{(top)}&(\sigma_1(x_1^{(1)}),  \ldots,
\sigma_1(x_{n_1}^{(1)}),\ldots, \sigma_m(x_1^{(m)}),\ldots,
\sigma_m(x_{n_m}^{(m)}); k, \epsilon, P_1,\ldots,P_{r_1})\\
&\subseteq \Gamma_R^{(top)}(\sigma_1(x_1^{(1)}), \ldots,
\sigma_1(x_{n_1}^{(1)}); k, \epsilon, P_1^{(1)},\ldots,P_r^{(1)})\oplus \\
&\qquad \qquad  \cdots \oplus\Gamma_R^{(top)}(
\sigma_m(x_1^{(m)}),\ldots,
\sigma_m(x_{n_m}^{(m)}); k, \epsilon, P_1^{(m)},\ldots,P_r^{(m)})\\
&=\Gamma_R^{(top)}( x_1^{(1)} ,  \ldots,  x_{n_1}^{(1)} ; k,
\epsilon, P_1^{(1)},\ldots,P_r^{(1)})\oplus \cdots \oplus
\Gamma_R^{(top)}( x_1^{(m)} ,  \ldots, x_{n_m}^{(m)} ; k, \epsilon,
P_1^{(m)},\ldots,P_r^{(m)}),
\end{aligned}
$$ where $$
P_j^{(i)}(\sigma_i(x_1^{(i)}), \ldots,
\sigma_i(x_{n_i}^{(i)}))=P_j(0,\ldots, 0,\sigma_i(x_1^{(i)}),
\ldots, \sigma_i(x_{n_i}^{(i)}), 0, \ldots, 0), \ \ 1\le j\le r,
1\le i\le m.$$
 By the definition of topological free entropy
dimension, we get that
$$
\delta_{top}(\sigma_1(x_1^{(1)}), \ldots,
\sigma_1(x_{n_1}^{(1)}),\ldots, \sigma_m(x_1^{(m)}),\ldots,
\sigma_m(x_{n_m}^{(m)}))\le \sum_{i=1}^m
\delta_{top}(x_1^{(i)},\ldots, x_{n_i}^{(i)}).
$$

\end{proof}

\subsection{ } Suppose that $\mathcal A$ is a unital C$^*$-algebra and $x_1,\ldots, x_n$ is a family of self-adjoint elements in $\mathcal A$.
  Recall Voiculescu's semi-microstates as follows. {\em Suppose
that $\{Q_r\}_{r=1}^\infty$ is a family of noncommutative
polynomials in $\Bbb C\langle X_1,\ldots, X_n\rangle$ with rational
coefficients. Let $R,\epsilon>0$, $r,k\in \Bbb N$. Define
$$
\Gamma_R^{(top1/2))}(x_1,\ldots, x_n;k,\epsilon,Q_1,\ldots,Q_r)
$$
to be the subset of $(\mathcal M_k^{s.a}(\Bbb C))^{n }$ consisting
of all these $$ (A_1,\ldots, A_n )\in (\mathcal M_k^{s.a}(\Bbb
C))^{n}
$$ satisfying
 $$  \max \{\|A_1\|, \ldots, \|A_n\| \}\le R
 $$ and
$$
 \|Q_j(A_1,\ldots, A_n )\|\le \| Q_j(x_1,\ldots, x_n )\| + \epsilon,
\qquad \forall \ 1\le j\le r.
$$}
 It is easy to see that
$$
\Gamma_R^{(top)}(x_1,\ldots, x_n;k,\epsilon,Q_1,\ldots,Q_r)\subseteq
\Gamma_R^{(top1/2)}(x_1,\ldots, x_n;k,\epsilon,Q_1,\ldots,Q_r).
$$

\begin{lemma}
 Suppose $\{\mathcal A_i\}_{i=1}^m$ ($m\ge 2$) is a family of unital
C$^*$-algebras  and   $\mathcal D$  is the full free product of the
unital C$^*$-algebras  $\{\mathcal A_i\}_{i=1}^m$  equipped with the
unital
  embedding  $\{ \sigma_i:   \mathcal A_i\rightarrow
\mathcal D\}_{i=1}^m$. Suppose $\{x_j^{(i)}\}_{j=1}^{n_i}$ is a
family of self-adjoint generators of $\mathcal A_i$ for $1\le i\le
m$.   Let $R>\max\{\|x_j^{(i)}\|,\ 1\le i\le m, 1\le j\le n_i\}$ be
a positive number.
 For any $r_0\in\Bbb N$ and $\epsilon_0>0$, there are $r_1\in \Bbb N$ and
 $\epsilon_1>0$ such that, for
 any $k\in \Bbb N$, if
 $$
   \begin{aligned}
     (A_1^{(i)}, \ldots, A_{n_i}^{(i)}) \in \Gamma_R^{(top1/2)}( x_1^{(i)},
     \ldots,x_{n_i}^{(i)}; k, \epsilon_1,P_1^{(i)},\ldots,
     P_{r_1}^{(i)}), \  \text { for } \ 1\le i\le m,
   \end{aligned}
 $$ where $P_1^{(i)},\ldots,
     P_r^{(i)} $ is defined as in Lemma 5.1,  then $$
\begin{aligned}  (A_1^{(1)},& \ldots, A_{n_1}^{(1)},\ldots, A_1^{(m)}, \ldots, A_{n_m}^{(m)})\notag\\ & \in
\Gamma_R^{(top1/2)}(\sigma_1(x_1^{(1)}),\ldots,\sigma_1(x_{n_1}^{(1)}),
\cdots,\sigma_m(x_1^{(m)}),\ldots,\sigma_m(x_{n_m}^{(m)}) ; k,
\epsilon_0,P_1,\ldots, P_{r_0}).
\end{aligned} $$
\end{lemma}
\begin{proof}
 We will prove the result by using the contradiction. Suppose, to
 the contrary, the result does not hold. Then  there are some $r_0\in \Bbb N$ and $\epsilon_0>0$
 so that the following holds: \begin{enumerate} \item []
  for any $r\in \Bbb
 N$, there are $k_r\in \Bbb N$ and
 $$
   \begin{aligned}
     (A_1^{(i,r)}, \ldots, A_{n_i}^{(i,r)}) \in \Gamma_R^{(top1/2)}
     ( x_1^{(i)},\ldots,x_{n_i}^{(i)}; k_r, 1/r,P_1^{(i)},\ldots,
     P_r^{(i)}), \  \text { for } \ 1\le i\le m,
   \end{aligned}
 $$ satisfying
\begin{align}  (A_1^{(1,r)},& \ldots, A_{n_1}^{(1,r)},\ldots, A_1^{(m,r)}, \ldots, A_{n_m}^{(m,r)})\notag\\ & \notin
\Gamma_R^{(top1/2)}(\sigma_1(x_1^{(1)}),\ldots,\sigma_1(x_{n_1}^{(1)}),
\cdots,\sigma_m(x_1^{(m)}),\ldots,\sigma_m(x_{n_m}^{(m)}) ; k_r,
\epsilon_0,P_1,\ldots, P_{r_0}).
\end{align}
\end{enumerate}

Let $\gamma$ be a free ultra-filter in $\beta(\Bbb N)\setminus
\Bbb N$. Let $\prod_{r=1}^\gamma \mathcal M_{k_r}(\Bbb C)$ be the
C$^*$ algebra ultra-product of matrices algebras $(\mathcal
M_{k_r}(\Bbb C))_{r=1}^\infty$ along the ultra-filter $\gamma$,
i.e. $\prod_{r=1}^\gamma \mathcal M_{k_r}(\Bbb C)$ is the quotient
algebra of the unital C$^*$-algebra $\prod_r^\infty \mathcal
M_{k_r}(\Bbb C)$ by  $\mathcal I_\infty $, where $\mathcal
I_\infty= \{(Y_r)_{r=1}^\infty \in \prod_r \mathcal M_{k_r}(\Bbb
C) \ | \ \lim_{r\rightarrow \gamma} \|Y_r\| =0\} $.

Let $\phi_i$ be the unital $*$-homomorphism from the C$^*$-algebra
$\mathcal A_i$ into the C$^*$-algebra $\prod_{r=1}^\gamma \mathcal
M_{k_r}(\Bbb C)$, induced by the mapping
$$
   x_j^{(i)} \rightarrow [(A_j^{(i,r)})_r]\in \prod_{r=1}^\gamma \mathcal M_{k_r}(\Bbb C), \qquad \forall \ 1\le j\le n_i,
$$ where $[(A_j^{(i,r)})_r]$ is the image of
$(A_j^{(i,r)})_{r=1}^\infty$ in the quotient algebra
$\prod_{r=1}^\gamma \mathcal M_{k_r}(\Bbb C)$.

By the definition of full free product, we know that there is a
unital $*$-homomorphism $\psi$ from $\mathcal D$ into
$\prod_{r=1}^\gamma \mathcal M_{k_r}(\Bbb C)$ so that $\phi_i=
\psi\circ \sigma_i.$ Hence,
$$\begin{aligned}
  \lim_{r\rightarrow \gamma} &\|P_t(A_1^{(1,r)},\ldots,
  A_{n_1}^{(1,r)},\ldots, A_1^{(m,r)},\ldots,
  A_{n_m}^{(m,r)})\|\\&  \le \|P_t(\sigma_1( x_1^{(1)}),\ldots,
 \sigma_1( x_{n_1}^{(1)}),\ldots, \sigma_m(x_1^{(m)}),\ldots,
  \sigma_m(x_{n_m}^{(m)}))\|,\qquad \forall \ 1\le t\le r_0.
\end{aligned}$$    This contradicts with the fact (5.1). This completes the proof of
the lemma.
\end{proof}

Recall the definition of a stable family of elements in a unital
C$^*$-algebra in \cite {HaSh3} as follows.

\begin{definition}
Suppose that $\mathcal A$ is a unital C$^*$-algebra and $x_1,\ldots,
x_n$ is a family of self-adjoint elements in $\mathcal A$. Suppose
that $\{Q_r\}_{r=1}^\infty$ is a family of noncommutative
polynomials in $\Bbb C\langle X_1,\ldots, X_n\rangle$ with rational
coefficients. The family of elements $x_1,\ldots,x_n$ is called
stable if for any $\alpha<\delta_{top}(x_1,\ldots,x_n)$ there are
positive numbers $R, C>0$ and $\omega_0>0$, $r_0\ge 1$, $k_0\ge 1$
so that
$$
\nu_\infty(\Gamma_{R}^{(top)}(x_1,\ldots, x_n;q\cdot k_0,\frac 1 r,
Q_1,\ldots, Q_r), \omega)\ge C^{ (q\cdot k_0)^2}\left( \frac 1
\omega\right)^{\alpha \cdot  (q\cdot k_0)^2}, \forall \
0<\omega<\omega_0, r>r_0, q\in \Bbb N.
$$
\end{definition}

\begin{example}  Any family of
self-adjoint generators $ x_1,\ldots, x_n $ of a finite dimensional
C$^*$-algebra   is
 stable.
A self-adjoint element $x$ in a unital C$^*$-algebra is stable. (see
\cite {HaSh3})
\end{example}

\subsection{Main result in this  section}

Now we are ready to show the additivity of topological free
entropy dimension in the full free products of some unital
C$^*$-algebras. In this subsection, we will use the following
notation.

\begin{notation}
 Suppose that $A\in \mathcal M_{k_1}(\Bbb C)$ and $B\in \mathcal M_{k_2}(\Bbb C)$. We denote the element
  $$\left ( \begin{aligned}
  A  & \quad 0\\
  0 &\quad B
  \end{aligned} \right )\in \mathcal M_{k_1+k_2}(\Bbb C)$$ by $A\oplus B$.
  \end{notation}
The concept of MF algebras was introduced by Blackadar and Kirchberg
in \cite{BK}.  It plays an important role in the classification of
C$^*$-algebras and it is connected to the question whether the
extension, in the sense of Brown, Douglas and Fillmore, of a unital
C$^*$-algebra is a group (see the striking result of Haagerup and
Thorb{\o}rnsen on $Ext(C_r^*(F_2)$).

The following result, which follows  directly from Theorem 5.2 in
the Appendix,  shows that Voiculescu's topological free entropy
dimension is well defined on a finite generated unital MF algebras.
\begin{lemma}
Suppose that $\mathcal A$ is a finitely generated  unital Blackadar
and Kirchberg's MF C$^*$-algebra. Then $\mathcal A$ has the
approximation property in the sense of Definition 4.2.
\end{lemma}


Now we are ready to prove our main result in this section.
\begin{theorem}
  Suppose that $\{\mathcal A_i\}_{i=1}^m$ $(m\ge 2)$  is a family of  unital C$^*$-algebras.
  Suppose that $\{x_j^{(i)}\}_{j=1}^{n_i}$ is
   a family of self-adjoint
  generators of $\mathcal A_i$ for $i=1,2,\ldots,m$. Suppose 
   $\{x_j^{(i)}\}_{j=1}^{n_i}$
 is a
 stable  family in the sense of Definition 5.2 for  $1\le i\le
 m$.
  Let the unital C$^*$-algebra $\mathcal
  D$ be the full free product of  $\{\mathcal A_i\}_{i=1}^m$
   equipped with  the unital embedding $\{ \sigma_i: \ \mathcal A_i\rightarrow \mathcal D\}_{i=1}^m$. Then
$$
\delta_{top}(\sigma_1(x_1^{(1)}), \ldots,
\sigma_1(x_{n_1}^{(1)}),\ldots, \sigma_m(x_1^{(m)}),\ldots,
\sigma_m(x_{n_m}^{(m)}))= \sum_{i=1}^m
\delta_{top}(x_1^{(i)},\ldots, x_{n_i}^{(i)}).
$$
If we identify each $x_j^{(i)}$ in $\mathcal A_i$ with its image
$\sigma_i(x_j^{(i)})$ in $\mathcal D$ when no confusion arises, then
$$
\delta_{top}( x_1^{(1)} , \ldots,  x_{n_1}^{(1)} ,\ldots,
 x_1^{(m)} ,\ldots,  x_{n_m}^{(m)} )= \sum_{i=1}^m
\delta_{top}(x_1^{(i)},\ldots, x_{n_i}^{(i)}).
$$
\end{theorem}

\begin{proof}
By  Lemma 5.1, we need only to prove that $$
\delta_{top}(\sigma_1(x_1^{(1)}), \ldots,
\sigma_1(x_{n_1}^{(1)}),\ldots, \sigma_m(x_1^{(m)}),\ldots,
\sigma_m(x_{n_m}^{(m)}))\ge \sum_{i=1}^m
\delta_{top}(x_1^{(i)},\ldots, x_{n_i}^{(i)}).
$$

Let $R>\max_{1\le i\le n, 1\le j\le n_i}  \|x_j^{(i)}\|  $ be a
positive
  number.
Since for all $1\le i\le m$,   $\{x_j^{(i)}\}_{j=1}^{n_i}$
 is a
 stable  family in the sense of Definition 5.2, each  $\mathcal A_i$ is a     unital MF C$^*$
algebras. By Theorem 5.4 in the appendix,
  the full free product $\mathcal D$ is   also an  unital MF C$^*$-algebra. By Lemma 5.3, it follows that for  any
   $r_0\in\Bbb N$  , there is some $k_0\in \Bbb N$ so
   that
   $$
\Gamma_R^{(top)}(\sigma_1(x_1^{(1)}), \ldots,
\sigma_1(x_{n_1}^{(1)}),\ldots, \sigma_m(x_1^{(m)}),\ldots,
\sigma_m(x_{n_m}^{(m)}); k_0, 1/r_0, P_1,\ldots,P_{r_0})\ne
\emptyset.
   $$

Note that for all $1\le i\le m$,   $\{x_j^{(i)}\}_{j=1}^{n_i}$
 is a
 stable  family. By Definition 5.2, for all $$
\alpha_i<\delta_{top}(x_1^{(i)},\ldots, x_{n_i}^{(i)}), \qquad
i=1,2,\ldots, m,
$$
there are positive numbers $C>0$ and $\omega_0>0$, $r_1\ge r_0$,
$k_1\ge 1$ so that
 \begin{align}
&\nu_\infty(\Gamma_{R}^{(top)}(x_1^{(i)},\ldots,
x_{n_i}^{(i)};q\cdot k_1,\frac 1 r, P_1^{(i)},\ldots, P_r^{(i)}),
\omega)\ge C^{ (q\cdot k_1)^2}\left( \frac 1
\omega\right)^{\alpha_i \cdot  (q\cdot k_1)^2},
\end{align}
  for $  \ 0<\omega<\omega_0, \ r>r_1, \ q\in \Bbb N,$ and $  1 \le i\le m$.

By Lemma 5.2,  when $r$ is large enough, for all $q\in \Bbb N$  we
have that if
 \begin{align}
 (A_1^{(1)} ,   &\ldots, A_{n_1}^{(1)} ,\ldots, A_1^{(m)} ,\ldots,
A_{n_m}^{(m)} ) \qquad \qquad \qquad \qquad \notag \\ &\in
\Gamma_R^{(top)}(\sigma_1(x_1^{(1)}), \ldots,
\sigma_1(x_{n_1}^{(1)}),\ldots, \sigma_m(x_1^{(m)}),\ldots,
\sigma_m(x_{n_m}^{(m)}); k_0, 1/r_0, P_1,\ldots,P_{r_0})
\end{align}
and
  \begin{align}
(B_1^{(i)}, \ldots, B_{n_i}^{(i)}) \in \Gamma_R^{(top1/2)}(
x_1^{(i)},\ldots,x_{n_i}^{(i)}; qk_1, \epsilon,P_1^{(i)},\ldots,
     P_r^{(i)}), \  \text { for } \ 1\le i\le m ,\end{align}
then
 \begin{align}
&(A_1^{(1)}\oplus  B_1^{(1)}  ,   \ldots, A_{n_1}^{(1)} \oplus
B_{n_1}^{(1)} ,\ldots, A_1^{(m)}\oplus B_1^{(m)} ,\ldots,
A_{n_m}^{(m)}\oplus B_{n_m}^{(m)})\notag  \\
& \in \Gamma_R^{(top)}(\sigma_1(x_1^{(1)}),  \ldots,
\sigma_1(x_{n_1}^{(1)}),\ldots, \sigma_m(x_1^{(m)}),\ldots,
\sigma_m(x_{n_m}^{(m)}); k_0+q  k_1 , 1/r_0, P_1,\ldots,P_{r_0}).
\end{align}  On the other hand, by (5.2), for $  \
0<\omega<\omega_0/2$, $r $ large enough,
 \begin{align}
\nu_\infty(  \Gamma_R^{(top1/2)}&(x_1^{(i)},\ldots, x_{n_i}^{(i)};q
k_1, 1/r,P_1^{(i)},\ldots, P_r^{(i)}  ), 2\omega)\notag \\ &\ge
\nu_\infty(\Gamma_{R}^{(top)}(x_1^{(i)},\ldots, x_{n_i}^{(i)};q
k_1,\frac 1 r, P_1^{(i)},\ldots, P_r^{(i)}), 2\omega)\notag \\& \ge
C^{ (q\cdot k_1)^2}\left( \frac 1 {2\omega}\right)^{\alpha_i \cdot
(q\cdot k_1)^2}, \quad \forall \ q\in \Bbb N.\end{align} Hence by
the relationship between packing number and covering number, we know
there is a family of elements
 \begin{align}
\{(B_1^{(\lambda_i)}, \ldots, B_{n_i}^{(\lambda_i)})
\}_{\lambda_i\in \Lambda_i}\subset \Gamma_R^{(top1/2)}(
x_1^{(i)},\ldots,x_{n_i}^{(i)}; qk_1, \epsilon,P_1^{(i)},\ldots,
     P_r^{(i)}), \  \text { for } \ 1\le i\le m ,\end{align}
where $\Lambda_i$ is an index set satisfying
\begin{align}
|\Lambda_i|\ge \nu_\infty(  \Gamma_R^{(top1/2)}&(x_1^{(i)},\ldots,
x_{n_i}^{(i)};q k_1, 1/r,P_1^{(i)},\ldots, P_r^{(i)}  ), 2\omega)
\ge C^{ (q\cdot k_1)^2}\left( \frac 1 {2\omega}\right)^{\alpha_i
\cdot (q\cdot k_1)^2};
\end{align}
and
\begin{align}
\|(B_1^{(\lambda_i)}, \ldots, B_{n_i}^{(\tilde
\lambda_i)})-(B_1^{(\tilde \lambda_i)}, \ldots,
B_{n_i}^{(\tilde\lambda_i)})\|\ge \omega, \qquad \forall\
\lambda_i\ne \tilde \lambda_i \in \Lambda_i.
\end{align}
 Thus,
combining (5.3), (5.4), (5.5), (5.7), (5.8) and (5.9) we have
$$\begin{aligned}
\nu_\infty(  \Gamma_R^{(top)}(\sigma_1(x_1^{(1)}),  \ldots,
 &\sigma_1(x_{n_1}^{(1)}),\ldots,  \sigma_m(x_1^{(m)}),\ldots,
\sigma_m(x_{n_m}^{(m)}); k_0+q  k_1 , 1/r_0, P_1,\ldots,P_{r_0}),
\omega)\\ &  \ge \prod_{i=1}^m |\Lambda_i|\ge \prod_{i=1}^m\left (
C^{ (q\cdot k_1)^2}\left( \frac 1 {2\omega}\right)^{\alpha_i \cdot
(q\cdot k_1)^2}\right),\qquad \forall \ q\in \Bbb N.\end{aligned}
$$
 This induces that
$$
\delta_{top}(\sigma_1(x_1^{(1)}),  \ldots,
\sigma_1(x_{n_1}^{(1)}),\ldots, \sigma_m(x_1^{(m)}),\ldots,
\sigma_m(x_{n_m}^{(m)}))\ge \sum_{i=1}^m \alpha_i,
$$whence $$
\delta_{top}(\sigma_1(x_1^{(1)}),  \ldots,
\sigma_1(x_{n_1}^{(1)}),\ldots, \sigma_m(x_1^{(m)}),\ldots,
\sigma_m(x_{n_m}^{(m)}))\ge\sum_{i=1}^m
\delta_{top}(x_1^{(i)},\ldots, x_{n_i}^{(i)}) .
$$ By Lemma 5.1, we get
$$
\delta_{top}(\sigma_1(x_1^{(1)}),  \ldots,
\sigma_1(x_{n_1}^{(1)}),\ldots, \sigma_m(x_1^{(m)}),\ldots,
\sigma_m(x_{n_m}^{(m)}))=\sum_{i=1}^m \delta_{top}(x_1^{(i)},\ldots,
x_{n_i}^{(i)}).
$$
Or
$$
\delta_{top}( x_1^{(1)} , \ldots,  x_{n_1}^{(1)} ,\ldots,
 x_1^{(m)} ,\ldots,  x_{n_m}^{(m)} )= \sum_{i=1}^m
\delta_{top}(x_1^{(i)},\ldots, x_{n_i}^{(i)}),
$$ when no confusion arises.
\end{proof}

As a corollary, we have the following result.
\begin{corollary}
Suppose that $\mathcal A_i$ ($i=1,2,\ldots, m$) is a unital C$^*$
algebra generated by a self-adjoint element $x_i$  in $\mathcal
A_i$. Let $\mathcal D$ be the full free product of $\mathcal A_1,
\ldots \mathcal A_n$ equipped with    unital embedding from each
$\mathcal A_i$ into $\mathcal D$. Identify the element $x_i$ in
$\mathcal A_i$ with its image in $\mathcal D$. Then
$$
\delta_{top}(x_1,\ldots, x_n) = \sum_{i=1}^n
\delta_{top}(x_i)=n-\sum_{i=1}^n \frac 1 {n_i},
$$ where $n_i$ is the number of elements in the spectrum of
$x_i$ in $\mathcal A_i$. (We use the notation $1/\infty=0$)

\end{corollary}
\begin{proof}
It follows from   Example 5.1, Theorem 5.1 and the results in
\cite{HaSh2}.
\end{proof}
\begin{corollary}
Suppose that $\mathcal A_i$ is a finite dimensional  C$^*$-algebra
generated by a family of self-adjoint element $\{x_{j}^{(i)}\}_{1\le
j\le n_i}$ for $1\le i\le n$. Let $\mathcal D$ be the full free
product of $\mathcal A_1, \ldots \mathcal A_n$ equipped with unital
embedding from each $\mathcal A_i$ into $\mathcal D$. Identify the
element $\{x_{j}^{(i)}\}$ in $\mathcal A_i$ with its image in
$\mathcal D$. Then
$$
\delta_{top}(\{x_{j}^{(i)}\}_{1\le j\le n_i, 1\le i\le n}) =
\sum_{i=1}^n \delta_{top}(\{x_{j}^{(i)}\}_{1\le j\le n_i})=
n-\sum_{i=1}^n \frac 1{dim_{\Bbb C} \mathcal A_i},
$$ where $dim_{\Bbb C}
\mathcal A_i$ is the complex dimension of $\mathcal A_i$.
\begin{proof}
It follows from   Example 5.1 , Theorem 5.1 and the results in
\cite{HaSh3}.
\end{proof}
\end{corollary}

\newpage

\centerline{\Large\bf{Appendix: Full Free Product of Unital MF
C$^*$-algebras  }}

\vspace {0.5cm}

\centerline{\Large\bf{ is \ \  MF Algebra}}

\vspace{1cm}

The concept of MF algebras was introduced by Blackadar and Kirchberg
in \cite{BK}. This     class of C$^*$-algebras is of interest for
many reasons. For example, it plays an important role in the
classification of C$^*$-algebras and it is connected to the question
whether the extension semigroup, in the sense of Brown, Douglas and
Fillmore, of a unital C$^*$-algebra is a group (see the striking
result of Haagerup and Thorb{\o}rnsen on $Ext(C_r^*(F_2)$). Thanks
to Voiculescu's result in \cite{VoiQusi}, we know that every
quasidigonal C$^*$-algebra is an MF algebra. Many   properties of MF
algebras have been discussed in \cite{BK}.  For example, it was
shown there that the inductive limit of MF algebras is an MF algebra
and every subalgebra of an MF algebra is an MF algebra. In this
appendix, we will prove that {\em unital full free product of two
unital separable MF algebras is, again, an MF algebra.}

Let us fix  notation first. We always assume that $\mathcal H$ is a
separable complex Hilbert space and $B(\mathcal H)$ is
 the set of all bounded operators on $\mathcal H$. Suppose $ \{x, x_k\}_{k=1}^\infty$ is a
  family of elements in $B(\mathcal H)$.
 We say $x_k\rightarrow x$ in $*$-SOT ($*$-strong operator topology) if and only if $x_k\rightarrow x$ in SOT and
 $x_k^*\rightarrow x^*$ in  SOT. Suppose $\{x_1,\ldots, x_n\} $ and
 $\{x_1^{(k)},\ldots, x_n^{(k)}\}_{k=1}^\infty$ are families of
 elements in $B(\mathcal H)$. We say
 $$
\langle x_1^{(k)},\ldots, x_n^{(k)} \rangle \rightarrow  \langle
x_1,\ldots, x_n \rangle, \ in \ *-SOT, \ \ as \ k \rightarrow \infty
 $$ if and only if
 $$
x_i^{(k)}   \rightarrow    x_i \ in \ *-SOT, \ \ as \ k \rightarrow
\infty, \ \ \forall \ 1\le i\le n.
 $$

Suppose $\{\mathcal A_k\}_{k=1}^\infty$ is a family of unital
C$^*$-algebras. Let $\prod \mathcal A_k$ be C$^*$-direct product of
the $\mathcal A_k$, i.e. the set of bounded sequences $(
x_k)_{k=1}^\infty$, with $x_k\in \mathcal A_k$, with pointwise
operations and sup norm; and let $\sum \mathcal A_k$ be the
C$^*$-direct sum, the set of sequences converging to zero in norm.
Then $\prod \mathcal A_k$ is a C$^*$-algebra and $\sum \mathcal A_k$
is a closed two-sided ideal; let $\pi$ be the quotient map from
$\prod \mathcal A_k$ to $\prod \mathcal A_k/\sum \mathcal A_k$. Then
$\prod \mathcal A_k/\sum \mathcal A_k$ is a unital C$^*$-algebra. If
we denote $\pi(( x_k)_{k=1}^\infty)$ by
 $[(x_k)_k]$ for any $( x_k)_{k=1}^\infty$ in $\prod \mathcal A_k$, then
 $$
 \|[(x_k)_k]\|= \limsup_{k\rightarrow \infty} \|x_k\|.
 $$

Suppose $\mathcal A$ is a  separable  unital C$^*$-algebra on a
Hilbert space $\mathcal H$. Let $\mathcal H^\infty=\oplus_{\Bbb N}
\mathcal H$, and for any $x\in \mathcal A$, let $x^\infty$ be the
element $ \oplus_{\Bbb N} x =(x,x,x,\ldots )$ in $\prod \mathcal
A^{(k)}\subset B(\mathcal H^\infty)$, where $\mathcal A^{(k)}$ is
the $k$-th   copy of $\mathcal A$.

Suppose  $\mathcal A$ is a separable unital C*-algebra and $\pi_i :
\mathcal A \rightarrow B(\mathcal H_i)$ are unital
$*$-representations for $i = 1, 2$. If there is a  unitary $U:
\mathcal H_1\rightarrow \mathcal H_2$, $n=1,2,\cdots$, satisfying
$$
 U^*  \pi_2(x)U=\pi_1(x) ,\ \ \forall \ x \in \mathcal A,
$$ then we say that $\pi_1$ and $\pi_2$ are   unitarily
equivalent, which is denoted by $\pi_1  \thickapprox_u \pi_2$.
Furthermore, if $\mathcal A$ is generated by a family of
self-adjoint elements $\{x_1,\ldots, x_n\}$, then $\pi_1
\thickapprox_u \pi_2$ will also be denoted by
$$
\langle \pi_1(x_1),\ldots, \pi_1(x_n)\rangle \thickapprox_u \langle
\pi_2(x_1),\ldots, \pi_2(x_n)\rangle.
$$

Suppose $\mathcal A$ is a separable unital C*-algebra and $\pi_i :
\mathcal A \rightarrow B(\mathcal H_i)$ are unital
$*$-representations for $i = 1, 2$. If there is a sequence of
unitaries $U_n: \mathcal H_1\rightarrow \mathcal H_2$,
$n=1,2,\cdots$, satisfying
$$
\|U^*_n \pi_2(x)U_n -\pi_1(x)\|\rightarrow 0, \ \ as \ \ n\rightarrow \infty,\ \ \forall \ x \in \mathcal A,
$$then we say that $\pi_1$ and $\pi_2$ are approximately unitarily
equivalent, which is denoted by $\pi_1\thicksim_a \pi_2$. Recall the following important result  by Voiculescu:
 \begin{lemma} Let $\mathcal A$ be a separable unital C*-algebra and
 $\pi_i : \mathcal A \rightarrow B(\mathcal H_i)$ be unital faithful
 $*$-representations for $i = 1, 2$ satisfying $\pi_i(\mathcal A)\cap \mathcal K(\mathcal H_i)=0$ for $i=1,2$, where $\mathcal K(\mathcal H_i)$ is the set of compact operators on $\mathcal H_i$. Then
    $\pi_1\thicksim_a \pi_2.$ \end{lemma}

The following lemma will be needed in the proof of Theorem 5.2.

\begin{lemma}
Suppose that $\mathcal H_1$ and $\mathcal H_2$ are two separable
infinite dimensional complex Hilbert spaces. Let
$\{e_n\}_{n=1}^\infty$ be an orthonormal  basis of $\mathcal H_1$
and $\{f_n\}_{n=1}^\infty$ be an orthonormal  basis of $\mathcal
H_2$. (Thus $\{e_n\oplus 0, 0 \oplus f_m\}_{m,n=1}^\infty$ is an
orthonormal  basis of $\mathcal H_1\oplus \mathcal H_2$). Let
$$
T=\begin{pmatrix} A & B\\
C& D \end{pmatrix} \ \in B(\mathcal H_1\oplus \mathcal H_2),
$$ where $A\in B(\mathcal H_1), B \in B(\mathcal H_2, \mathcal H_1), C
 \in B(\mathcal H_1, \mathcal H_2)$ and
$D \in B(\mathcal H_2).$ Assume $\{U_n\}_{n=1}^\infty$ is a
family of unitary operators from $\mathcal H_2$ to $\mathcal
H_1\oplus \mathcal H_2$ so that
$$
   U_n(f_i)= 0\oplus f_i, \ \ \forall \ 1\le i\le n.
$$ Then the following statements are true:
\begin{enumerate}
  \item [(a)] $U_n^*TU_n\rightarrow D$ in the weak operator topology.
  \item [(b)] If $B=0$ and $C=0$, then $U_n^*TU_n\rightarrow D$ in the $*$-strong operator topology.
\end{enumerate}
\end{lemma}
\begin{proof}
It is an easy exercise.
\end{proof}
The following results gives some equivalent definitions of an MF algebra.
\begin{theorem}
Suppose that $\mathcal A$ is a unital C$^*$-algebra  generated by
a family of self-adjoint elements $x_1,\ldots, x_n$ in $\mathcal
A$. Then the following are equivalent:
\begin{enumerate}
   \item $\mathcal A$ is an MF algebra, i.e. there is a unital embedding from $\mathcal
   A$ into $\prod_k \mathcal M_{n_k}(\Bbb C)/\sum \mathcal M_{n_k}(\Bbb
   C)$ for a sequence of positive integers $\{n_k\}_{k=1}^\infty $.
   \item $\mathcal A$ has approximation property in the sense of Definition 4.2;
   \item There are a sequence of positive integers $\{m_k\}_{k=1}^\infty$ and  self-adjoint matrices
   $ A_1^{(k)},\ldots, A_n^{(k)} $ in $\mathcal M_{m_k}^{s.a.}(\Bbb C)$ for $k=1,2,\ldots$, such that
   $$
   \lim_{k \rightarrow \infty}\|P(A_1^{(k)},\ldots, A_n^{(k)})\|= \|P(x_1,\ldots,x_n)\|,   \ \forall \ P\in \Bbb C\langle X_1,\ldots,X_n\rangle,
   $$ where $\Bbb C\langle
X_1,\ldots, X_n\rangle $ is the set of all  noncommutative
polynomials in the indeterminates $X_1,\ldots, X_n$.
   \item Suppose $\pi: \mathcal A \rightarrow B(\mathcal H)$ is a faithful $*$-representation of $\mathcal A$ on an infinite dimensional separable complex Hilbert space $\mathcal H$. Then there are a sequence of positive integers $\{m_k\}_{k=1}^\infty$, families of self-adjoint matrices
   $\{A_1^{(k)},\ldots, A_n^{(k)}\}$ in $\mathcal M_{m_k}^{s.a.}(\Bbb C)$ for $k=1,2,\ldots$, and unitary operators $U_k: \mathcal
   H \rightarrow  (\Bbb C^{m_k})^\infty$ for $k=1,2,\ldots$,  such that
  \begin{enumerate} \item $$
   \lim_{k \rightarrow \infty}\|P(A_1^{(k)},\ldots, A_n^{(k)})\|= \|P(x_1,\ldots,x_n)\|,   \ \forall \ P\in \Bbb C\langle X_1,\ldots,X_n\rangle;
   $$\item $$
      U_k^*\ (A_i^{(k)})^\infty  \ U_k  \rightarrow \pi(x_i) \ in \ *-SOT  \ as \ k \rightarrow \infty,\ \ for \ 1\le i\le n,
   $$ where $(A_i^{(k)})^\infty=  A_i^{(k)}\oplus A_i^{(k)}\oplus A_i^{(k)}  \cdots \in B((\Bbb C^{m_k})^\infty)$. \end{enumerate}
\end{enumerate}

\end{theorem}

\begin{proof} $(1)\Leftrightarrow(2)\Leftrightarrow(3)$ are directly from Theorem 3.2.2 in \cite {BK}, Definition 5.3 and Lemma 5.6 in \cite{HaSh2}. $(4)\Rightarrow (3)$ is trivial. We need only to prove $(3) \Rightarrow (4)$.

Let $\{P_r\}_{r=1}^\infty$ be the collection of all noncommutative
polynomials in $\Bbb C\langle X_1,\ldots, X_n\rangle $ with rational
complex coefficients. Let $\{\xi_r\}_{r=1}^\infty$ be a dense subset
of the unit ball in $\mathcal H$. To show that $(3) \Rightarrow
(4)$, we need only to prove the following:  {\em Let $\pi: \mathcal
A \rightarrow B(\mathcal H)$ be a faithful $*$-representation of
$\mathcal A$ on an infinite dimensional separable complex Hilbert
space $\mathcal H$. For any positive integer $k$, there are a
positive integer  $ m_k $, a family of self-adjoint matrices
   $\{A_1^{(k)},\ldots, A_n^{(k)}\}$ in $\mathcal M_{m_k}^{s.a.}(\Bbb C)$, and a unitary
   operator  $U_k: \mathcal
   H \rightarrow  (\Bbb C^{m_k})^\infty$,  such that
  \begin{enumerate}\item  [(a)]  $$
   |\|P_r(A_1^{(k)},\ldots, A_n^{(k)})\|- \|P_r(x_1,\ldots,x_n)\||< \frac 1 k,   \ \forall \ 1\le r\le k;
   $$\item [(b)] $$
     \| U_k^*\ (A_i^{(k)})^\infty  \ U_k \cdot \xi_r -\pi(x_i)\cdot \xi_r\|<\frac 1 k,\ \ for \ 1\le i\le n, \ 1\le r\le k.
   $$ \end{enumerate}

 }

 Assume that $(3)$ is true. Thus there are a sequence of positive integers $\{t_s\}_{s=1}^\infty$ and families of self-adjoint matrices
   $\{B_1^{(s)},\ldots, B_n^{(s)}\}$ in $\mathcal M_{t_s}^{s.a.}(\Bbb C)$ for $s=1,2,\ldots$, such that
   $$
   \lim_{s \rightarrow \infty}\|P(B_1^{(s)},\ldots, B_n^{(s)})\|= \|P(x_1,\ldots,x_n)\|,   \ \forall \ P\in \Bbb C\langle X_1,\ldots,X_n\rangle.$$

 For any positive integers $N_1<N_2$ and
$1\le i\le n$, we define
$$
  \begin{aligned}
     D(N_1,N_2;i) & = B_i^{({N_1})}\oplus B_i^{({N_1+1})}\oplus \cdots\oplus B_i^{({N_2})}\\
     D(N_1;i) & = B_i^{({N_1})}\oplus B_i^{({N_1+1})}\oplus B_i^{({N_1+2})}\oplus \cdots\\
      D(N_1,N_2;i)^\infty &= D(N_1,N_2;i)\oplus D(N_1,N_2;i)\oplus D(N_1,N_2;i)\oplus \cdots\\
       D(N_1;i)^\infty &= D(N_1;i)\oplus D(N_1;i)\oplus D(N_1;i)\oplus
       \cdots.
  \end{aligned}
$$
It is not hard to see that, for any $P\in \Bbb C\langle
X_1,\ldots,X_n\rangle$,
 \begin{align}
\|P(D(N_1,N_2;1)^\infty ,\cdots, D(N_1,N_2;n)^\infty )\|&= \|P(D(N_1,N_2;1) ,\cdots, D(N_1,N_2;n) )\| \\
\|P(D(N_1;1)^\infty ,\cdots, D(N_1;n)^\infty )\|&= \|P(D(N_1;1) ,\cdots, D(N_1;n) )\| \\
\lim_{N_2\rightarrow \infty} \|P(D(N_1,N_2;1) ,\cdots, D(N_1,N_2;n) )\| &= \|P(D(N_1;1) ,\cdots, D(N_1;n) )\|\\
\|P(D(N_1;1)^\infty ,\cdots, D(N_1;n)^\infty )\| &=\sup_{N_2\ge N_1} \|P(B_1^{(N_2)},
\ldots, B_n^{(N_2)})\| \notag\\
&\ge \|P(x_1,\ldots,x_n)\| \\
\lim_{N_1\rightarrow \infty}\|P(D(N_1;1) ,\cdots, D(N_1;n) )\|& =
\|P(x_1,\ldots,x_n)\|.
\end{align}
Furthermore, we let
$$ \begin{aligned}
E(N_1,N_2;i)&=(B_i^{({N_1})})^\infty\oplus
(B_i^{({N_1+1})})^\infty\oplus
(B_i^{({N_1+2})})^\infty\oplus\cdots \oplus(B_i^{({N_2})})^\infty \\
E(N_1;i)&=(B_i^{({N_1})})^\infty\oplus
(B_i^{({N_1+1})})^\infty\oplus (B_i^{({N_1+2})})^\infty\oplus
\cdots.
\end{aligned}
$$It is easy to see that
\begin{align}
 \langle E(N_1,N_2;1),\ldots, E(N_1,N_2;n)\rangle &\thickapprox_u
\langle
 D(N_1,N_2;1)^\infty,\ldots,  D(N_1,N_2;n)^\infty\rangle \\
   \langle E(N_1 ;1),\ldots, E(N_1 ;n)\rangle &\thickapprox_u \langle
 D(N_1 ;1)^\infty,\ldots,  D(N_1 ;n)^\infty\rangle.
\end{align}

 Let $\mathcal D(N_1)=C^*(D(N_1;1)^\infty ,\cdots, D(N_1;n)^\infty )$
 be the unital C$^*$-algebra generated by $$D(N_1;1)^\infty , \cdots, D(N_1;n)^\infty $$ on
 a separable Hilbert space $\mathcal H_1$. Let $id_{\mathcal D(N_1)}$ be the
  identity representation of $\mathcal D(N_1)$ on $\mathcal H_1$, i.e.
   $$id_{\mathcal D(N_1)}: \mathcal D(N_1)\rightarrow \mathcal D(N_1)$$
    is defined by $id_{\mathcal D(N_1)}(x)=x$ for any $x$ in ${\mathcal D(N_1)}$. By the inequality (5.13), we know there is a
    unital $*$-homomorphism $$\rho_{N_1}: \mathcal D(N_1)\rightarrow \pi(\mathcal A)$$ such that
 $$
 \rho_{N_1}(D(N_1;i)^\infty )=\pi(x_i), \ \ \forall \ 1\le i\le n,
 $$ because $\pi$ is a faithful $*$-representation of $\mathcal A$ on
 $\mathcal H$.
Denote
$$
  \begin{aligned}
\vec D(N_1) &= \langle
 D(N_1 ;1)^\infty,\ldots,  D(N_1 ;n)^\infty\rangle \\
 \vec x &= \langle x_1,\ldots, x_n\rangle.
  \end{aligned}
$$
By Lemma 5.4, we know that
$$
id_{\mathcal D(N_1)} \thicksim_a id_{\mathcal D(N_1)}\oplus \rho_{N_1}.
$$
It follows that,
$$\begin{aligned}
\vec D(N_1) \oplus \pi(\vec x)&=\langle
 D(N_1 ;1)^\infty\oplus \pi(x_1),\ldots,  D(N_1 ;n)^\infty\oplus \pi(x_n)\rangle\\&=(id_{\mathcal
D(N_1)} \oplus \rho_{N_1})(\vec D(N_1))  \\
& \in \overline{ \left \{W^* \ \vec D(N_1) \ W \ | \ W: \mathcal
H_1\oplus \mathcal H\rightarrow \mathcal H_1 \text{ is a uniatry
operator} \right\}}^{\ Norm}.\end{aligned}
$$
By Lemma 5.5, we know that
$$
\pi(\vec x) \in \overline{ \left \{V^* \ \vec D(N_1)\oplus \pi(\vec
x) \ V \ | \ V: \mathcal H\rightarrow \mathcal H_1\oplus \mathcal H
\text{ is a uniatry operator} \right\}}^{\ *-SOT}.
$$
It induces that
\begin{equation}
\pi(\vec x) \in \overline{ \left \{U_1^* \ \vec D(N_1) \ U_1 \ | \
U_1: \mathcal H\rightarrow \mathcal H_1  \text{ is a uniatry
operator} \right\}}^{\
*-SOT}.
\end{equation}

Take a positive integer $k$. For such   $k$ and  the family of
polynomials $\{P_1,\ldots, P_k\}$, by equation (5.11), equation
(5.14), we know that there is  a positive integer $N_1$ such that
 \begin{enumerate}\item  [(i)]  $$
   |\|P_r(D(N_1 ;1) ,\ldots,  D(N_1 ;n) )\|- \|P_r(x_1,\ldots,x_n)\||< \frac 1 {4k},   \ \forall \ 1\le r\le
   k.
   $$ 
   \end{enumerate}
For the  family of vectors $\{\xi_1,\ldots, \xi_k\}$, by   the fact
(5.17)  we know there is a unitary $U_1: \mathcal H\rightarrow
\mathcal H_1$, such that
\begin{enumerate}
  \item [(ii)] $$
     \| U_1^*\    D(N_1 ;i)^\infty  \ U_1\cdot \xi_r -  \pi(x_i) \cdot \xi_r\|<\frac 1 {4k},\ \ for \ 1\le i\le n, \ 1\le r\le k.
   $$ \end{enumerate} In view of the fact  (5.16), i.e.
$$
  \langle
 D(N_1 ;1)^\infty,\ldots,  D(N_1 ;n)^\infty\rangle   \thickapprox_u\langle E(N_1 ;1),\ldots, E(N_1
 ;n)\rangle
$$
and  the fact (5.11), together with the definitions of
$E(N_1,N_2;i)$ and $E(N_1;i)$, we know   there are a positive
integer $N_2>N_1$ and a unitary $U_2: \mathcal H\rightarrow \mathcal
H_3$ such that
\begin{enumerate}\item  [(iii)]  $$
   |\|P_r(E(N_1,N_2 ;1) ,\ldots,  E(N_1,N_2 ;n) )\|- \|P_r(x_1,\ldots,x_n)\||< \frac 1 {2k},   \ \forall \ 1\le r\le k;
   $$\item [(iv)] $$
     \| U_2^*\    E(N_1,N_2;i)  \ U_2 \cdot \xi_r -  \pi(x_i) \cdot \xi_r\|<\frac 1 {2k},\ \ for \ 1\le i\le n, \ 1\le r\le
     k,
   $$ \end{enumerate} where we assume that $E(N_1,N_2 ;1) ,\ldots,  E(N_1,N_2
   ;n) $ act on a separable Hilbert space $\mathcal H_3$.
   In view of the fact (5.15), i.e.
$$\langle E(N_1,N_2;1),\ldots, E(N_1,N_2;n)\rangle  \thickapprox_u
\langle
 D(N_1,N_2;1)^\infty,\ldots,  D(N_1,N_2;n)^\infty\rangle,$$
if we let
   $$ A_i^{(k)} = D(N_1,N_2;i) \in\mathcal M_{m_k}^{s.a.}(\Bbb C), \ \ 1\le i\le n $$ for some $m_k\in \Bbb
   N$, then there is some
    $U_k:
   \mathcal H \rightarrow  (\Bbb C^{m_k})^\infty$ such that
     \begin{enumerate}\item  [(a)]  $$
   |\|P_r(A_1^{(k)},\ldots, A_n^{(k)})\|- \|P_r(x_1,\ldots,x_n)\||< \frac 1 k,   \ \forall \ 1\le r\le k;
   $$\item [(b)] $$
     \| U_k^*\ (A_i^{(k)})^\infty  \ U_k \cdot \xi_r -\pi(x_i)\cdot \xi_r\|<\frac 1 k,\ \ for \ 1\le i\le n, \ 1\le r\le k.
   $$ \end{enumerate}
   This completes the proof of the theorem.
\end{proof}

Recall that $\mathcal A\subset B(\mathcal H)$ is a separable
 quasidiagonal C$^*$-algebra  if there  is an increasing
sequence of finite-rank projections $\{E_i\}_{i=1}^\infty$ on $H$
tending strongly to the identity such that $\| x E_i-E_i x
\|\rightarrow 0$ as $i\rightarrow \infty$ for any $ x \in \mathcal
A.$  The examples of quasidiagonal C$^*$-algebras include all
abelian
 C$^*$-algebra and all finite dimensional C$^*$-algebras.  
 An abstract separable
C$^*$-algebra $\mathcal A$ is quasidiagonal if there is a faithful
$*$-representation  $\pi: \mathcal A\rightarrow B(\mathcal H)$ such
that $\pi(\mathcal A)\subset B(\mathcal H)$ is quasidiagonal.

Following lemma will be needed in the proof of Theorem 5.3.
\begin{lemma}
Suppose that $\mathcal A\subset B(\mathcal H)$ is a separable unital
quasidiagonal  C$^*$-algebra, and $x_1,\ldots, x_n$ are self-adjoint
elements in $\mathcal A$. For any $\epsilon>0$, any finite subset
$\{P_1,\ldots, P_r\}$ of $\Bbb C\langle X_1,\ldots, X_n\rangle$ and
any finite subset $\{\xi_1\ldots, \xi_r\}$ of $\mathcal H$, there is
a finite rank projection $p $ in $B(\mathcal H)$ such that:
\begin{enumerate}
\item [(i)] $\| \xi_k-p\ \xi_k\|<\epsilon $, $\|p\  x_ip\ \xi_k-x_i\xi_k\|<\epsilon, $ for all $1\le
i\le n$ and $1\le k\le r$;
\item [(ii)] $|\|P_j(p \ x_1p , \ldots, p \ x_np )\|_{B(p\mathcal H)}-\|P_j( x_1 , \ldots,
 x_n )\| |<\epsilon,$ for all $1\le j\le r.$ \end{enumerate}
\end{lemma}
\begin{proof}
It follows directly from the definition of quasidiagonality of a
separable  C$^*$-algebra.
\end{proof}

\begin{theorem}
Suppose that $\mathcal H$ is a separable complex Hilbert space,
  and $  r \in \Bbb N$.    Suppose that
 $$
 \{P_1,\ldots, P_r\}\subset \Bbb C\langle X_1,\ldots, X_n\rangle;
 \quad \{Q_1,\ldots, Q_r\}\subset \Bbb C\langle Y_1,\ldots,
 Y_m\rangle
 $$ are families of noncommutative polynomials. Suppose that
 $$\{\xi_1,\ldots, \xi_r\}\subset \mathcal H$$ is a family of unit vectors in $\mathcal H$.
Suppose that $ \mathcal A \subset B(\mathcal H)$ and  $ \mathcal B
\subset B(\mathcal H)$ are two separable unital quasidiagonal
C$^*$-algebras on $\mathcal H$. Assume that $\mathcal A $, or
$\mathcal B $, is generated by a family of self-adjoint elements
$x_1 , \ldots, x_n $ in $\mathcal A $, or by a family of
self-adjoint elements $y_1 , \ldots, y_m $ in $\mathcal B $
respectively.  Then there are  a positive integer $t$,   a family of
self-adjoint matrices
$$
\{A_1 , \ldots, A_n, \ B_1 , \ldots, B_m \}\subset \mathcal
M_{t}^{s.a.}(\Bbb C),
$$ and unitary $U :   \mathcal
   H   \rightarrow (\Bbb C^{t })^\infty $ such hat
   \begin{enumerate}
\item
$$ |\|P_i(A_1 , \ldots, A_n )\|-\|P_i(x_1 ,\ldots,
x_n \||<\frac 1 r; \ \ \forall \ 1\le i\le r
$$ and
$$ |\|Q_i(B_1 , \ldots, B_m )\|-\|Q_i(y_1 ,\ldots,
y_m \||<\frac 1 r; \ \ \forall \ 1\le i\le r
$$
\item
$$
\| U^*A_i^\infty U\xi_k - x_i\xi_k\|<\frac 1 r, \ \ \forall \ 1\le
i\le n, \ 1\le k\le r;
$$and
$$
\| U^* B_j^\infty U \xi_k - y_j\xi_k\|<\frac 1 r, \ \ \forall \ 1\le
j\le m, \ 1\le k\le r,
$$ where we follow the previous notation by assuming
$$
A_i^\infty= A_i\oplus A_i\oplus \cdots; \quad B_j^\infty=
B_j\oplus B_j\oplus \cdots
$$ are operators acting on the Hilbert space  $(\Bbb C^t)^\infty=  \Bbb C^t\oplus \Bbb C^t
\oplus \Bbb C^t \oplus \cdots$.
   \end{enumerate}

\end{theorem}

\begin{proof} Without the loss of generality, we will assume that
$\|x_i\|\le 1$ and $\|y_j\|\le 1$ for all $1\le i\le n, 1\le j\le
m$. Note that $\mathcal A$ and $\mathcal B$ are two separable
unital quasidiagonal C$^*$-algebras on $\mathcal H$. By Lemma
5.6, there are finite rank projections $p$ and $q$ in $B(\mathcal
H)$ such that:
\begin{enumerate}
\item [(i)] for all $1\le i\le n$, $1\le j\le m$ and $1\le k\le r$, we have $$\| p \
\xi_k- \xi_k\|<\frac 1 {2r}, $$
  and $$\| q\ \xi_k- \xi_k\|<\frac 1{2 r}; $$
\item [(ii)] for all $1\le i\le r$,   we have $$|\|P_i(px_1p, \ldots, px_np)\|-\|P_i( x_1 , \ldots,
 x_n )\||<\frac 1 r,$$   and
$$|\|Q_i(qy_1q, \ldots, qy_mq)\|-\|Q_i( y_1 , \ldots,
 y_m )\||<\frac 1 r.$$
 \end{enumerate}
Let $\tilde p$ be a finite rank projection in $B(\mathcal H)$ such
that   $p\le \tilde p$, $q\le\tilde p$  and   rank($\tilde p$) is
a common multiple of rank($  p$) and rank($q$). Therefore there
are equivalent mutually  orthogonal projections $e_1=p$,
$e_2,\ldots, e_{d_1}$ and equivalent mutually  orthogonal
projections $f_1=q, f_2,\ldots, f_{d_2}$ such that
$\sum_{i=1}^{d_1}e_i=\sum_{j=1}^{d_2}f_j=\tilde p$. Let
$v_1,\ldots,v_{d_1}$ and $w_1,\ldots, w_{d_2}$ be partial
isometries in $B(\mathcal H)$ such that $v_i^*v_i=e_1=p$,
$v_iv_i^*=e_i$, $w_j^*w_j=f_1=q$, and $w_jw_j^*=f_j$. Let
$\mathcal K=\tilde p\mathcal H$ and
$$
  \begin{aligned}
    \tilde x_i & = \sum_{i_1=1}^{d_1} v_{i_1}   x_i   v_{i_1}^*, \ \
    \ 1\le i\le n\\
    \tilde y_j & = \sum_{j_1=1}^{d_2} w_{j_1}   y_j   w_{j_1}^*, \ \
    \ 1\le j\le m.
  \end{aligned}
$$It is easy to check that
\begin{enumerate}
\item [(iii)] for all $1\le i\le n$, $1\le j\le m$ and $1\le k\le r$,
we have $$\|\tilde x_i \ p \xi_k-x_i\xi_k\|\le
   \|\tilde x_i \ p \  \xi_k-x_ip\xi_k\|
+\| x_ip\xi_k- x_i  \xi_k\| <\frac 1 r, $$
  and $$\| \tilde y_j\ q\xi_k-y_j\xi_k\|
\le   \|\tilde y_j \ q \  \xi_k-y_jq\xi_k\| +  \| y_jq\xi_k- \ y_j
\xi_k\| <\frac 1 r;
$$
\item [(iv)] for all $1\le i\le r$,   we have $$|\|P_i(\tilde x_1, \ldots, \tilde x_n)\|-\|P_i( x_1 , \ldots,
 x_n )\|| = |\|P_i(px_1p, \ldots, px_np)\|-\|P_i( x_1 , \ldots,
 x_n )\||<\frac 1 r,$$   and
$$|\|Q_i( \tilde y_1, \ldots,  \tilde y_m)\|-\|Q_i( y_1 , \ldots,
 y_m )\||=|\|Q_i(qy_1q, \ldots, qy_mq)\|-\|Q_i( y_1 , \ldots,
y_m)\||<\frac 1 r.$$
 \end{enumerate}

 Let $t=dim_{\Bbb C} \mathcal K$ and $\rho: B(\mathcal K)\rightarrow
 \mathcal M_t(\Bbb C)$ be a $*$-isomorphism from $B(\mathcal K)$
 onto $ \mathcal M_t(\Bbb C)$. Let
 $$
    \begin{aligned}
      A_i = \rho(\tilde x_i), \ \ \ B_j= \rho(\tilde y_j) \ \ \ for
      \  \ 1\le i\le n, \ 1\le j\le m;
    \end{aligned}
 $$ and $$
A_i^\infty= A_i\oplus A_i\oplus \cdots; \quad B_j^\infty=
B_j\oplus B_j\oplus \cdots
$$ be operators acting on the Hilbert space  $(\Bbb C^t)^\infty=  \Bbb C^t\oplus \Bbb C^t
\oplus \Bbb C^t \oplus \cdots$. Thus $$A_i=A_i^*,  \ \ \ B_j=B_j^*,
\qquad 1\le i\le n, \  1\le j\le m.$$ Moreover there exists a
unitary $U : \mathcal
   H  \rightarrow (\Bbb C^{t })^\infty   $, which induces a
faithful  $*-$representation  $Ad(U^*)$ of $ \mathcal M_t(\Bbb
C)^\infty$ on $\mathcal H$ satisfying  $\rho^{\infty} = Ad(U).
   $ Thus by (iii) and (iv), we have
    \begin{enumerate}
\item
$$ |\|P_i(A_1 , \ldots, A_n )\|-\|P_i(x_1 ,\ldots,
x_m \||<\frac 1 r, \ \ \forall \ 1\le i\le r
$$ and
$$ |\|Q_i(B_1 , \ldots, B_m )\|-\|Q_i(y_1 ,\ldots,
y_m \||<\frac 1 r; \ \ \forall \ 1\le i\le r;
$$
\item  for $ \ 1\le
k\le r, \ 1\le i\le n,  \ 1\le k\le r,$
$$
\| U^*A_i^\infty U\xi_k - x_i\xi_k\|=\|(U^*A_i^\infty U
\xi_k-U^*A_i^\infty U p\xi_k)+(\tilde x_i p\xi_k-x_i\xi_k)\|<\frac 2
r;
$$and
$$
\| U^* B_j^\infty U \xi_k - y_j\xi_k\|=\|(U^*B_j^\infty U
\xi_k-U^*B_j^\infty U  q\xi_k)+(\tilde y_jq\xi_k-y_j\xi_k)\|<\frac 2
r.
$$ \end{enumerate} This completes the proof of the theorem.
\end{proof}

Recall the definition of the unital full free product of two unital
C$^*$-algebras as follows.

\begin{definition}
The unital full free product of   unital C$^*$-algebras $\{\mathcal
A_i\}_{i=1}^m$ ($m\ge 2$) is a unital C$^*$-algebra $\mathcal D$,
denoted by $\mathcal A_1*_{\Bbb C } \mathcal A_2*_{\Bbb C } \cdots
*_{\Bbb C } \mathcal A_m$, equipped with
 unital embedding  $\{ \sigma_i: \ \mathcal A_i\rightarrow \mathcal D\}_{i=1}^m$,  such that (i) the set
$\cup_{i=1}^m \sigma_i(\mathcal A_i) $ is norm dense in $\mathcal
D$; and (ii) if $\phi_i$   is a unital $*$-homomorphism from
$\mathcal A_i$  into a unital C$^*$-algebra $\bar{\mathcal D}$ for
$i=1,2,\ldots, m$, then there is a unital $*$-homomorphisms $\psi $
from $\mathcal D$ to $\bar{\mathcal D}$ satisfying $\phi_i=\psi
\circ \sigma_i$, for $i=1,2,\ldots,m$. If there is no confusion
arising, we will identify the element $x_i$ in $\mathcal A_i$ with
its image in $\mathcal D$.
\end{definition}

Suppose that $\mathcal A$ and $\mathcal B$ are unital C$^*$-algebras
with a family of self-adjoint generators $ x_1,\ldots, x_n $, and $
y_1,\ldots, y_m $ respectively. Let $\mathcal A*_{\Bbb C } \mathcal
B$ be the unital full free product of $\mathcal A$ and $\mathcal B$,
defined as above, with a family of self-adjoint generators $\{
x_1,\ldots, x_n,y_1,\ldots, y_m\}$. Let   $\{P_r\}_{r=1}^\infty$, or
$\{Q_r\}_{r=1}^\infty$, or $\{\Psi_r\}_{r=1}^\infty$  be the
collection of all noncommutative polynomials in $\Bbb C\langle
X_1,\ldots, X_n\rangle $, or $\Bbb C\langle Y_1,\ldots, Y_m\rangle
$, or $\Bbb C\langle X_1,\ldots, X_n,Y_1,\ldots, Y_m\rangle $
respectively,  with rational complex coefficients. Then we have the
following result.
\begin{lemma}
 Let $\mathcal A, \mathcal B$ and $x_1,\ldots, x_n, y_1,\ldots, y_m$
 be as above. For any positive integer $a$, there is a positive
 integer $r$ so that the following hold:
   \begin{enumerate} \item []
   If $A_1, \ldots, A_n, B_1,\ldots, B_m$ are self-adjoint elements
   on a separable Hilbert space $\mathcal K$ so that
 $$ |\|P_i(A_1 , \ldots, A_n )\|-\|P_i(x_1 ,\ldots,
x_n) \||<\frac 1 r, \ \ \forall \ 1\le i\le r
$$ and
$$ |\|Q_i(B_1 , \ldots, B_m )\|-\|Q_i(y_1 ,\ldots,
y_m) \||<\frac 1 r, \ \ \forall \ 1\le i\le r,
$$ then, $\ \ \forall \ 1\le j\le {a},$
$$
 \|\Psi_j(A_1 , \ldots, A_n, B_1 , \ldots, B_m )\|_{B(\mathcal
K)}< \|\Psi_j(x_1 ,\ldots, x_n ,y_1 ,\ldots, y_m )\|_{\mathcal
A*_{\Bbb C } \mathcal B}+\frac 1 {a}.
$$

   \end{enumerate}
\end{lemma}
\begin{proof}
We will prove the result of the lemma by using contradiction. Assume
that there are a positive integer $a$ and a sequence of self-adjoint
elements
$$A_1^{(r)}, \ldots, A_n^{(r)}, B_1^{(r)},\ldots, B_m^{(r)} \ \ \ in \ \ \ B(\mathcal
K^{(r)}),\ \ \ r=1,2,\ldots$$  satisfying, for all $r\ge 1$,
 \begin{equation} |\|P_i(A_1^{(r)} , \ldots, A_n^{(r)} )\|-\|P_i(x_1 ,\ldots,
x_n) \||<\frac 1 {r}, \ \ \forall \ 1\le i\le { r} \end{equation}
and
 \begin{equation} |\|Q_i(B_1^{(r)} , \ldots, B_m^{(r)} )\|-\|Q_i(y_1 ,\ldots,
y_m) \||<\frac 1 { r}, \ \ \forall \ 1\le i\le r;
\end{equation} but
\begin{equation}
  \max_{  1\le j\le {a} }\{ \|\Psi_j(A_1^{(r)} , \ldots, A_n^{(r)}, B_1^{(r)} ,
\ldots, B_m^{(r)} )\|_{B(\mathcal K^{(r)})}- \|\Psi_j(x_1 ,\ldots,
x_n ,y_1 ,\ldots, y_m )\|_{\mathcal A*_{\Bbb C } \mathcal B}\}\ge
\frac 1 {a} .
\end{equation}
Consider the unital C$^*$-algebra $\prod_r B(\mathcal
K^{(r)})/\sum_r B(\mathcal K^{(r)})$ and elements
$$[(A_1^{(r)})_r] , \ldots, [(A_n^{(r)})_r], [(B_1^{(r)})_r] ,
\ldots, [(B_m^{(r)})_r] \ \ \in \ \ \prod_r B(\mathcal
K^{(r)})/\sum_r B(\mathcal K^{(r)}).$$ By the inequalities (5.18)
and (5.19), there are unital embedding $\phi_1: \mathcal A
\rightarrow  \prod_r B(\mathcal K^{(r)})/\sum_r B(\mathcal K^{(r)})$
and $\phi_2: \mathcal B \rightarrow  \prod_r B(\mathcal
K^{(r)})/\sum_r B(\mathcal K^{(r)})$ so that
$$
\phi_1(x_{i_1}) = [(A_{i_1}^{(r)})_r], \ \  \phi_2(y_{i_2}) =
[(B_{i_2}^{(r)})_r], \ \ 1\le {i_1}\le n, \ 1\le {i_2}\le m.
$$ By the definition of the full free product
$\mathcal A*_{\Bbb C } \mathcal B$, there is a $*$-homomorphism
$\phi: \mathcal A*_{\Bbb C } \mathcal B\rightarrow \prod_r
B(\mathcal K^{(r)})/\sum_r B(\mathcal K^{(r)})$ so that
$$
\phi (x_{i_1}) = [(A_{i_1}^{(r)})_r], \ \  \phi (y_{i_2}) =
[(B_{i_2}^{(r)})_r], \ \ 1\le {i_1}\le n, \ 1\le {i_2}\le m,
$$ where   $x_1,\ldots, x_n, y_1,\ldots, y_m$ are identified  as the elements
in $\mathcal A*_{\Bbb C } \mathcal B$. Therefore, $\ \ \forall \
1\le j\le {a},$
$$\begin{aligned} \|\Psi_j([(A_1^{(r)})_r] ,  \ldots,
[(A_n^{(r)})_r], [(B_1^{(r)})_r] , &\ldots, [(B_m^{(r)})_r
])\|_{\prod_r B(\mathcal K^{(r)})/\sum_r B(\mathcal K^{(r)})}\\
&= \limsup_{r\rightarrow\infty } \|\Psi_j(A_1^{(r)} , \ldots,
A_n^{(r)}, B_1^{(r)} , \ldots, B_m^{(r)} )\|_{B(\mathcal K^{(r)})}\\
&\le \|\Psi_j(x_1 ,\ldots, x_n ,y_1 ,\ldots, y_m )\|_{\mathcal
A*_{\Bbb C } \mathcal B} ,\end{aligned}
$$ which contradicts   the inequality (5.20).

\end{proof}

In \cite{EX},    Exel and Loring showed that the unital full free
product of two residually finite dimensional C$^*$-algebra is
residually finite dimensional, which extends an earlier result by
Choi in \cite{Choi}. In \cite{Bo}, Boca showed that the unital full
free product of two quasidiagoanl C$^*$-algebras is also
quasidiagonal. Our next result provides the analogue of the
preceding results from Choi, Exel and Loring, and Boca in the
context of MF algebras.
\begin{theorem}
Suppose $\mathcal A$ and $\mathcal B$ are   finitely generated
unital MF algebras. Then the unital full free product $\mathcal
A*_{\Bbb C}\mathcal B$ is an MF algebra.

\end{theorem}

\begin{proof}
Assume that $\mathcal A$, and $\mathcal B$, are generated by a
family of self-adjoint elements $x_1,\ldots, x_n$, and $y_1,\ldots,
y_m$ respectively. Thus we can assume that $x_1,\ldots, x_n,
y_1,\ldots, y_m$ also generate  $\mathcal A*_{\Bbb C}\mathcal B$ as
a C$^*$-algebra. Assume the $\pi: \mathcal A*_{\Bbb C}\mathcal B
\rightarrow B(\mathcal H)$ is a faithful $*$-representation of
$\mathcal A*_{\Bbb C}\mathcal B$ on a separable Hilbert space
$\mathcal H$. Let $\{P_r\}_{r=1}^\infty$, or $\{Q_r\}_{r=1}^\infty$,
or $\{\Psi_r\}_{r=1}^\infty$  be the collection of all
noncommutative polynomials in $\Bbb C\langle X_1,\ldots, X_n\rangle
$, or $\Bbb C\langle Y_1,\ldots, Y_m\rangle $, or $\Bbb C\langle
X_1,\ldots, X_n,Y_1,\ldots, Y_m\rangle $ respectively,  with
rational complex coefficients.

To show that $\mathcal A*_{\Bbb C}\mathcal B$ is an MF algebra, it
suffices to show that {\em for any positive integer $r_0$, there are
a positive integer $t$ and  self-adjoint matrices
$$
A_1,\ldots, A_n, B_1,\ldots, B_m  \ \in \ \mathcal M_t(\Bbb C)
$$ such that, $ \forall \  1\le j\le r_0$,
$$
\left | \|\Psi_j(A_1 , \ldots, A_n, B_1 , \ldots, B_m )\|_{\mathcal
M_t(\Bbb C)}- \|\Psi_j(x_1 ,\ldots, x_n ,y_1 ,\ldots, y_m
)\|_{\mathcal A*_{\Bbb C } \mathcal B}\right |<\frac 1 {r_0}.
$$}

Suppose that $\{\xi_r\}$ is a dense subset of the unit ball of
$\mathcal H$. Note that $\mathcal A$ and $\mathcal B$ are MF
algebras. For any positive integer $r$, by Theorem 5.2 we know there
are   positive integers $n_r, m_r, $ and self-adjoint matrices
   $$ E_1^{(r)},\ldots, E_n^{(r)}  \in  \mathcal M_{n_r}^{s.a.}(\Bbb C), \ \
    F_1^{(r)},\ldots, F_m^{(r)}  \in  \mathcal M_{m_r}^{s.a.}(\Bbb C),
    $$   and unitary operators $U_r: \mathcal
   H \rightarrow  (\Bbb C^{n_r})^\infty$,  $V_r: \mathcal
   H \rightarrow  (\Bbb C^{m_r})^\infty$  such that
  \begin{enumerate} \item  \begin{align}
   \left |\|P_i(E_1^{(r)},\ldots, E_n^{(r)})\|- \|P_i(x_1,\ldots,x_n)\|\right |< \frac 1 {2r},   \ \forall \ 1\le i\le
   r;\\
   \left |\|Q_i(F_1^{(r)},\ldots, F_m^{(r)})\|- \|Q_i(y_1,\ldots,y_m)\|\right |< \frac 1 {2r},   \ \forall \ 1\le i\le
   r.
   \end{align}
   \item
     \begin{align} \|U_r^*\ (E_{i_1}^{(r)})^\infty
      \ U_r \xi_j  - \pi(x_{i_1})\xi_j \|<\frac 1 {2r},  \ \ \  \ \forall \ 1\le i_1\le
   n, \ 1\le j\le { r}; \\
   \|V_r^*\ (F_{i_2}^{(r)})^\infty
   \ V_r \xi_j - \pi(y_{i_2})\xi_j \|<\frac 1 {2r},  \ \ \  \ \forall \ 1\le i_2\le
   m, \ 1\le j\le { r}.
   \end{align}
    \end{enumerate}
    Note both $\{ U_r^*\ (E_{i_1}^{(r)})^\infty   \ U_r   \}_{i=1}^n$ and
    $\{V_r^*\ (F_{i_2}^{(r)})^\infty  \ V_r\}_{i_2=1}^m$ are sets of
    quasidiagonal operators on the Hilbert space $\mathcal H$. By
    Theorem 5.3,
there are a positive integer $t_r$,    families of self-adjoint
matrices
\begin{equation}
\{A_1^{(r)} , \ldots, A_n^{(r)}\} \subset \mathcal
M_{t_r}^{s.a.}(\Bbb C) \quad  \ \ \{B_1^{(r)} , \ldots, B_m^{(r)}
\}\subset \mathcal M_{t_r}^{s.a.}(\Bbb C), \end{equation} and
unitary $W_r : \mathcal
   H  \rightarrow (\Bbb C^{t_r })^\infty  $ such hat
   \begin{enumerate}
\item [(3)]
\begin{equation} |\|P_i(A_1^{(r)} , \ldots, A_n^{(r)}
)\|-\|P_i(E_1^{(r)},\ldots, E_n^{(r)})\||<\frac 1 {2r},  \ \
\forall \ 1\le i\le { r}, \end{equation} and
\begin{equation} |\|Q_i(B_1^{(r)} , \ldots, B_m^{(r)} )\|
-\|Q_i(F_1^{(r)},\ldots, F_m^{(r)})\||<\frac 1 {2r}, \ \ \forall
\ 1\le i\le r;
\end{equation}
\item [(4)]
\begin{equation}
\| W_r^*(A_{i_1}^{(r)})^\infty W_r\xi_j -U_r^*\
(E_{i_1}^{(r)})^\infty \ U_r \xi_j\|<\frac 1 {2r}, \ \ \forall \
1\le i_1\le n, \ 1\le j\le r;
\end{equation}and
\begin{equation}
\| W_r^* (B_{i_2}^{(r)})^\infty W_r \xi_j - V_r^*\
(F_{i_2}^{(r)})^\infty  \ V_r\xi_j\|<\frac 1 {2r}, \ \ \forall \
1\le i_2\le m, \ 1\le j\le r.
\end{equation}
   \end{enumerate}
Combining the inequalities (5.21), (5.22), (5,16) and (5.27), we
have
\begin{align}
   \left |\|P_i(A_1^{(r)} , \ldots, A_n^{(r)})\|- \|P_i(x_1,\ldots,x_n)\|\right |< \frac 1 r,   \ \forall \ 1\le i\le
   r;\\
   \left |\|Q_i(B_1^{(r)} , \ldots, B_m^{(r)})\|- \|Q_i(y_1,\ldots,y_m)\|\right |< \frac 1 r,   \ \forall \ 1\le i\le
   r.
   \end{align}
   Combining the inequalities (5.23), (5.24), (5.28) and (5.29), we
   have
\begin{align} \|W_r^*(A_{i_1}^{(r)})^\infty W_r \xi_j  - \pi(x_{i_1})\xi_j \|<\frac 1 r,  \ \ \  \ \forall \ 1\le i_1\le
   n, \ 1\le j\le r; \\
   \| W_r^* (B_{i_2}^{(r)})^\infty W_r \xi_j  - \pi(y_{i_2})\xi_j \|<\frac 1 r,  \ \ \  \ \forall \ 1\le i_2\le
   m, \ 1\le j\le r.
   \end{align}
By Lemma 5.7 and the inequalities (5.30), (5.31), we know that, for
$1\le j\le r_0$,
\begin{equation}
 \limsup_{r\rightarrow \infty} \|\Psi_j(A_1^{(r)} , \ldots,
 A_n^{(r)}, B_1^{(r)} , \ldots, B_m^{(r)} )\|_{\mathcal M_{t_r}(\Bbb
 C)} \le \|\Psi_j(x_1 ,\ldots, x_n ,y_1 ,\ldots, y_m )\|_{\mathcal
A*_{\Bbb C } \mathcal B}.
\end{equation}

Note  $\{W_r^*(A_{i_1}^{(r)})^\infty W_r , W_r^*
 (B_{i_2}^{(r)})^\infty W_r
\}_{r,i_1,i_2}$ is a bounded subset in $B(\mathcal H)$.
 Since $\{\xi_r\}_{r=1}^\infty$ is a dense subset of the unit
ball of $\mathcal H$, by the inequality (5.32) and (5.33) we know
that, as $r$ goes to infinity,
$$
 W_r^*(A_{i_1}^{(r)})^\infty W_r \rightarrow \pi(x_{i_1}), \ \ W_r^*
 (B_{i_2}^{(r)})^\infty W_r\rightarrow \pi(y_{i_2}) \ in \  SOT, \ \ \forall
 1\le i_1\le n, 1\le i_2\le m.
$$
Hence, for $1\le j\le r_0$
\begin{align}
 \liminf_{r\rightarrow \infty} &\|\Psi_j(A_1^{(r)} , \ldots,
 A_n^{(r)}, B_1^{(r)} , \ldots, B_m^{(r)} )\|_{\mathcal M_{t_r}(\Bbb
 C)}\notag\\
 &=\liminf_{r\rightarrow \infty}  \|\Psi_j(W_r^*(A_{ 1}^{(r)})^\infty W_r  , \ldots,
 W_r^*(A_{n}^{(r)})^\infty W_r , W_r^*
 (B_{1}^{(r)})^\infty W_r , \ldots, W_r^*
 (B_{m}^{(r)})^\infty W_r )\|_{B(\mathcal H)}\notag\\& \ge \|\Psi_j(x_1 ,\ldots, x_m ,y_1 ,\ldots, y_m )\|_{\mathcal
A*_{\Bbb C } \mathcal B}.
\end{align}
From the fact (5.25) and inequalities (5.34) and (5.35), it follows
that,  for the given $r_0$, there are a positive integer $t$ and
self-adjoint matrices
$$
A_1,\ldots, A_n, B_1,\ldots, B_m  \ \in \ \mathcal M_t(\Bbb C)
$$ such that, $ \forall \  1\le j\le r_0$,
$$
\left | \|\Psi_j(A_1 , \ldots, A_n, B_1 , \ldots, B_m )\|_{\mathcal
M_t(\Bbb C)}- \|\Psi_j(x_1 ,\ldots, x_n ,y_1 ,\ldots, y_m
)\|_{\mathcal A*_{\Bbb C } \mathcal B}\right |<\frac 1 {r_0}.
$$ This completes the proof of the theorem.
\end{proof}

\begin{remark}
{\em Using the similar argument in the proof of Theorem 5.3, the
result of Theorem 5.3 can be quickly extended as follows: }
Suppose $\mathcal A$ and $\mathcal B$ are separable  unital MF
algebras. Then the unital full free product $\mathcal A*_{\Bbb
C}\mathcal B$ is an MF algebra.

\end{remark}

In \cite{Haag},  Haagerup and  Thorbj{\o}rnsen showed  $C_r^*(F_2)$
is an MF algebra. Combining with Voiculescu's discussion in
\cite{VoiQusi}, they were able to conclude a striking result that
$Ext(C_r^*(F_2))$ is not a group. Also based on \cite{VoiQusi},
Brown showed in \cite{Br} that if $\mathcal  A$ is an  MF algebra
and $Ext(\mathcal A)$ is a group, then $\mathcal A$ is a
quasidiagonal C$^*$-algebra. It is a well-known fact that
$C_r^*(F_2)$ is not a quasidiagonal C$^*$-algebra  and any
subalgebra of a quasidiagonal C$^*$-algebra is again  quasidiagonal.
Now from Haagerup and Thorbj{\o}rnsen's result on $C_r^*(F_2)$  and
our Theorem 5.4, it quickly follows the next corollary.

\begin{corollary}Suppose that $\mathcal B$ is a unital separable MF
algebra. Then  $C_r^*(F_2)*_{\Bbb C}\mathcal B$ is an MF algebra.
Moreover, $Ext(C_r^*(F_2)*_{\Bbb C}\mathcal B)$ is not a group.
\end{corollary}

\bigskip

\vspace{1cm}

\noindent{Don Hadwin; \quad Qihui Li; \quad Junhao Shen }

\vspace{0.2cm}

\noindent{{Department of Mathematics and Statistics}}

 \noindent
University of New Hampshire}

\noindent  Durham, NH, 03824

\noindent {Email: don@math.unh.edu;
  \  qme2@cisunix.unh.edu; \
  \ jog2@cisunix.unh.edu \qquad }

\end{document}